\documentclass[12pt]{article}
\usepackage[]{amsmath,amssymb}
\usepackage{amscd}
\usepackage{amsthm}
\usepackage{latexsym}
\usepackage{cite}
\usepackage[all]{xy}

\newtheorem{definition}{Definition}[section]
\newtheorem{theorem}[definition]{Theorem}
\newtheorem{lemma}[definition]{Lemma}
\newtheorem{corollary}[definition]{Corollary}

\newtheorem{proposition}[definition]{Proposition}
\newtheorem{notation}[definition]{Notation}

\def\K{\mathbb K}

\def\X{\mathfrak X}

\renewcommand{\tilde}{\widetilde}
\renewcommand{\epsilon}{\varepsilon}
\begin{document}
\title{\bf
Cauchy Pairs and Cauchy Matrices
}
\author{
Alison Gordon Lynch}
\date{}

\maketitle
\begin{abstract}
 Let $\K$ denote a field and let $\X$ denote a finite non-empty set. Let $\text{Mat}_\X(\K)$ denote the $\K$-algebra consisting of the matrices with entries in $\K$ and rows and columns indexed by $\X$. A matrix $C \in \text{Mat}_\X(\K)$ is called {\it Cauchy} whenever there exist mutually distinct scalars $\{x_i\}_{i \in \X}, \{\tilde{x}_i\}_{i \in \X}$ from $\K$ such that $C_{ij} = (x_i - \tilde{x}_j)^{-1}$ for $i, j \in \X$.  In this paper, we give a linear algebraic characterization of a Cauchy matrix. To do so, we introduce the notion of a Cauchy pair.  A Cauchy pair is an ordered pair of diagonalizable linear transformations $(X, \tilde{X})$ on a finite-dimensional vector space $V$ such that $X-\tilde{X}$ has rank 1 and such that there does not exist a proper subspace $W$ of $V$ such that $X W \subseteq W$ and $\tilde{X} W \subseteq W$. Let $V$ denote a vector space over $\K$ with dimension $|\X|$. We show that for every Cauchy pair $(X, \tilde{X})$ on $V$, there exists an $X$-eigenbasis $\{v_i\}_{i \in \X}$ for $V$ and an $\tilde{X}$-eigenbasis $\{w_i\}_{i \in \X}$ for $V$ such that the transition matrix from $\{v_i\}_{i \in \X}$ to $\{w_i\}_{i \in \X}$ is Cauchy. We show that every Cauchy matrix arises as a transition matrix for a Cauchy pair in this way. We give a bijection between the set of equivalence classes of Cauchy pairs on $V$ and the set of permutation equivalence classes of Cauchy matrices in $\text{Mat}_\X(\K)$.

 \end{abstract}
\section{Introduction}

Throughout this paper, let $\K$ denote a field and let $\X$ denote a finite non-empty set. Let $\text{Mat}_\X(\K)$ denote the $\K$-algebra consisting of the matrices with entries in $\K$ and rows and columns indexed by $\X$.

\medskip
We recall the notion of a Cauchy matrix.
\begin{definition}\label{def:cm}
\rm
A matrix $C \in \text{Mat}_{\X}(\K)$ is called {\it Cauchy} whenever there exist mutually distinct scalars $\{x_i\}_{i \in \X}, \{\tilde{x}_i\}_{i\in\X}$ from $\K$ such that the $(i,j)$-entry of $C$ is
\begin{equation}\label{cauchymatrix}
    C_{ij} = \frac{1}{x_i - \tilde{x}_j} \qquad  (i,j\in \X).
\end{equation}
\end{definition}
Cauchy matrices have been studied in the context of rational interpolation problems \cite{gow1992}, efficient algorithms for solving systems of linear equations \cite{Pan2000}, error-correcting codes \cite{blomer1995}, \cite{schindelhauer2013}, and signal processing \cite{Bojanczyk1995}.

\medskip
In this paper, we give a linear algebraic characterization of a Cauchy matrix.  To do so, we introduce a linear algebraic object called a Cauchy pair. Roughly speaking, a Cauchy pair is an ordered pair of diagonalizable linear transformations $(X, \tilde{X})$ on a finite-dimensional vector space such that $X-\tilde{X}$ has rank 1 (see Definition \ref{def:cp} for the precise definition). We show that there exists (in a sense we will make precise) a 1-1 correspondence between Cauchy pairs and Cauchy matrices.

\medskip
We now describe our results in detail. Let $V$ denote a vector space over $\K$ with dimension $|\X|$ and let $(X, \tilde{X})$ be a Cauchy pair on $V$.  We show that each eigenspace of $X$ and each eigenspace of $\tilde{X}$ has dimension 1. Moreover, we show that $X$ and $\tilde{X}$ have no common eigenvalues.  Let $\{x_i\}_{i \in \X}$ (resp. $\{\tilde{x}_i\}_{i \in \X}$) denote the eigenvalues of $X$ (resp. $\tilde{X}$).  We say that $(\{x_i\}_{i \in \X}, \{\tilde{x}_i\}_{i \in \X})$ is {\it eigenvalue data} for $(X, \tilde{X})$. Let $C \in \text{Mat}_\X(\K)$ be the matrix with $(i,j)$-entry $(x_i - \tilde{x}_j)^{-1}$ for $i,j \in \X$. Note that $C$ is Cauchy. Therefore $C$ is invertible \cite[p. 38]{Horn2012}. We show that $C$ is the transition matrix from an $X$-eigenbasis for $V$ to an $\tilde{X}$-eigenbasis for $V$.

\medskip
We just obtained a Cauchy matrix from a Cauchy pair. We will show that every Cauchy matrix arises in this way. For any mutually distinct scalars $\{x_i\}_{i \in \X}, \{\tilde{x}_i\}_{i \in \X}$ from $\K$, we show that there exists a Cauchy pair $(X, \tilde{X})$ on $V$ with eigenvalue data $(\{x_i\}_{i \in \X}, \{\tilde{x}_i\}_{i \in \X})$.  Then the Cauchy matrix with $(i,j)$-entry $(x_i-\tilde{x}_j)^{-1}$ for $i, j \in \X$ is a transition matrix from an $X$-eigenbasis for $V$ to an $\tilde{X}$-eigenbasis for $V$.

\medskip
We now give a more detailed description of the role that a Cauchy matrix plays as a transition matrix for a Cauchy pair. Let $(X, \tilde{X})$ be a Cauchy pair on $V$ with eigenvalue data $(\{x_i\}_{i \in \X}, \{\tilde{x}_i\}_{i \in \X})$. We show that there exists an $X$-eigenbasis (resp. $\tilde{X}$-eigenbasis) for $V$ such that the sum of the basis elements is contained in the 1-dimensional subspace $(X-\tilde{X})V$. We call such a basis {\it $X$-standard} (resp. {\it $\tilde{X}$-standard}). We show that there exists a nonzero bilinear form $\langle \ , \ \rangle : V \times V \to \K$ such that $\langle X v, w\rangle = \langle v, Xw\rangle$ and $\langle \tilde{X} v, w \rangle = \langle v, \tilde{X}w \rangle$ for all $v, w \in V$. We call such a form {\it $(X, \tilde{X})$-invariant}.  We show that any $(X, \tilde{X})$-invariant form is symmetric, non-degenerate, and unique up to multiplication by a nonzero scalar in $\K$.  We show that the elements of an $X$-standard basis (resp. $\tilde{X}$-standard basis) for $V$ are mutually orthogonal with respect any $(X, \tilde{X})$-invariant form.  We define a {\it dual X-standard basis} (resp. {\it dual $\tilde{X}$-standard basis}) to be a basis for $V$ that is dual to an $X$-standard basis (resp. $\tilde{X}$-standard basis) for $V$ with respect to an $(X, \tilde{X})$-invariant form.  We compute the inner products and transition matrices between $X$-standard, $\tilde{X}$-standard, dual $X$-standard, and dual $\tilde{X}$-standard bases. We comment on one of the transition matrices. Let $\{v_i\}_{i \in \X}$ be an $X$-standard basis for $V$ and let $\{w_i\}_{i \in \X}$ be a dual $\tilde{X}$-standard basis for $V$. We show that the transition matrix from $\{v_i\}_{i \in \X}$ to $\{w_i\}_{i \in \X}$ is a nonzero scalar multiple of the Cauchy matrix with $(i,j)$-entry $(x_i - \tilde{x}_j)^{-1}$ for $i, j \in \X$.

\medskip
Below Definition \ref{def:cm}, we mentioned a 1-1 correspondence between Cauchy pairs and Cauchy matrices. We now make this correspondence precise. We say that Cauchy pairs $(X, \tilde{X})$ and $(Y, \tilde{Y})$ are {\it equivalent} whenever $(X, \tilde{X})$ is isomorphic to the Cauchy pair $(Y + \zeta I, \tilde{Y} + \zeta I)$ for some $\zeta \in \K$. We say that matrices $M$ and $N$ are {\it permutation equivalent} whenever there exist permutation matrices $P$ and $Q$ such that $M = PNQ$.  In our main result, we give a bijection between the following two sets:
\begin{itemize}
\item The equivalence classes of Cauchy pairs on $V$.
\item The permutation equivalence classes of Cauchy matrices in $\text{Mat}_\X(\K)$.
\end{itemize}
The bijection sends a Cauchy pair $(X, \tilde{X})$ to the Cauchy matrix with $(i,j)$-entry $(x_i - \tilde{x}_j)^{-1}$ for $i, j \in \X$, where $(\{x_i\}_{i \in \X}, \{\tilde{x}_i\}_{i \in \X})$ is eigenvalue data for $(X, \tilde{X})$.

\medskip
The paper is organized as follows. In Section 2 we discuss some preliminaries. In Sections 3,4 we define Cauchy pairs and discuss their properties and relationships. In Section 5 we discuss how Cauchy matrices are related to Cauchy pairs. In Sections 6--9 we discuss $X$-standard and $\tilde{X}$-standard bases, related scalars, and transition matrices. In Section 10 we introduce an $(X, \tilde{X})$-invariant bilinear form and investigate its properties. Moreover, we compute inner products for $X$-standard bases, $\tilde{X}$-standard bases, and their duals.  In Section 11 we compute the transition matrices between $X$-standard bases, $\tilde{X}$-standard bases, and their duals. In Section 12 we summarize our results on inner products and transition matrices in terms of matrices. In Section 13 we describe the bijective correspondence between Cauchy pairs and Cauchy matrices.

\section{Preliminaries}

In this section, we consider a system of linear equations that has connections to both Cauchy matrices and Cauchy pairs.  Let $\{a_i\}_{i\in\X}, \{b_i\}_{i\in\X}$ denote mutually distinct scalars from $\K$.  We consider the following system of linear equations in variables $\{\lambda_i\}_{i \in \X}$:
\begin{equation}\label{eq:system}
\sum_{i \in \X} \frac{\lambda_i}{a_i - b_j} = 1, \qquad j \in \X.
\end{equation}

We review the solution to $(\ref{eq:system})$. To do so, we recall the Lagrange polynomials \cite{Atkinson1989}.  Let $\lambda$ denote an indeterminate and let $\K[\lambda]$ denote the $\K$-algebra of polynomials in $\lambda$ that have all coefficients in $\K$.
\begin{definition}\label{def:lagrange}\rm \cite[p.132]{Atkinson1989}
\rm
Let $\{c_i\}_{i\in\X}, \{d_i\}_{i\in\X}$ denote scalars from $\K$ such that $\{c_i\}_{i \in \X}$ are mutually distinct.  There exists a unique polynomial $L \in \K[\lambda]$ of degree at most $|\X|-1$ such that $L(c_i) = d_i$ for $i \in \X$.  We call $L$ the {\it Lagrange polynomial} for $(\{c_i\}_{i \in \X}, \{d_i\}_{i \in \X})$.
\end{definition}

The following result is well known.
\begin{lemma}\label{lem:lagrange}{\rm\cite[p.134]{Atkinson1989}}
Let $\{c_i\}_{i\in\X}, \{d_i\}_{i\in\X}$ denote scalars from $\K$ such that $\{c_i\}_{i \in \X}$ are mutually distinct and let $L$ be the Lagrange polynomial for $(\{c_i\}_{i\in\X}, \{d_i\}_{i\in\X})$.  Then
\[
    L(\lambda) = \sum_{i \in \X} d_i \,\prod_{\substack{k \in \X \setminus i}} \frac{\lambda - c_k}{c_i - c_k}.
\]
\end{lemma}

\begin{lemma}\label{lem:lagrcoeff}
Let $\{c_i\}_{i\in\X}, \{d_i\}_{i\in\X}$ denote scalars from $\K$ such that $\{c_i\}_{i \in \X}$ are mutually distinct, and let $L$ be the Lagrange polynomial for $(\{c_i\}_{i\in\X}, \{d_i\}_{i\in\X})$.  Then the coefficient of $\lambda^{|\X|-1}$ in $L$ is
\[
\sum_{i \in \X} d_i \prod_{k \in \X \setminus i}\frac{1}{c_i - c_k}.
\]
\end{lemma}
\begin{proof}
Immediate from Lemma \ref{lem:lagrange}.
\end{proof}

\begin{lemma}\label{lem:lambdas}
For $i \in \X$, define $A_i \in \K$ by
\begin{equation}\label{eq:lambdas}
    A_i = \frac{\prod_{k \in \X} (a_i - b_k)}{\prod_{k \in \X \setminus i} (a_i - a_k)}.
\end{equation}
Then $\{A_i\}_{i \in \X}$ is the unique solution of the system of linear equations (\ref{eq:system}).
\end{lemma}
\begin{proof}
First, observe that the matrix of coefficients for the system of linear equations (\ref{eq:system}) is Cauchy and hence it is invertible.  So if $\{A_i\}_{i \in \X}$ is a solution to (\ref{eq:system}), then it is the unique solution.

\medskip
We now show that $\{A_i\}_{i \in \X}$ is a solution to (\ref{eq:system}). Fix $j \in \X$. We show that $\{A_i\}_{i\in \X}$ satisfies the $j$th equation of (\ref{eq:system}). That is, we show that
\begin{equation}\label{eq:lambdasum}
\sum_{i \in \X} \frac{\prod_{k \in \X \setminus j}(a_i - b_k)}{\prod_{k \in \X \setminus i} (a_i - a_k)} = 1.
\end{equation}
For $i \in \X$, let $d_i = \prod_{k \in \X\setminus j} (a_i - b_k)$. Let $L(\lambda)$ be the Lagrange polynomial for $(\{a_i\}_{i \in \X}, \{d_i\}_{i \in \X})$. By Lemma \ref{lem:lagrcoeff}, the coefficient of $\lambda^{|\X|-1}$ in $L(\lambda)$ is equal to the left side of (\ref{eq:lambdasum}). Thus, we can evaluate the left side of (\ref{eq:lambdasum}) by computing the coefficient of $\lambda^{|\X|-1}$ in $L(\lambda)$.

\medskip
Let $f(\lambda) = \prod_{k \in \X \setminus j}(\lambda-b_k)$.  Then $f(\lambda)$ has degree $|\X|-1$ and $f(a_i) = d_i$ for $i \in \X$.  By Definition \ref{def:lagrange}, this implies that $L(\lambda) = f(\lambda)$.  Thus, $L(\lambda)$ is monic of degree $|\X|-1$, so the coefficient of $\lambda^{|\X|-1}$ in $L(\lambda)$ is 1.  Therefore (\ref{eq:lambdasum}) holds and the result follows.
\end{proof}

We mention one significance of the system of linear equations (\ref{eq:system}). In order to describe this significance, we first give a definition and a lemma.

\begin{definition}\label{def:data}
\rm
Let $C \in \text{Mat}_{\X}(\K)$ be Cauchy and let $\{x_i\}_{i \in \X}, \{\tilde{x}_i\}_{i \in \X}$ denote mutually distinct scalars from $\K$.  We say that $(\{x_i\}_{i\in\X}, \{\tilde{x}_i\}_{i\in\X})$ is {\it data} for $C$ whenever $\{x_i\}_{i \in \X}, \{\tilde{x}_i\}_{i \in \X}$ satisfy $(\ref{cauchymatrix})$.
\end{definition}

\noindent It follows from Definition \ref{def:cm} and Definition \ref{def:data} that any two data for a Cauchy matrix are related in the following way.
\begin{lemma}\label{lem:equivdata}
Let $C \in \rm{Mat}_{\X}(\K)$ be Cauchy and let $(\{x_i\}_{i \in \X}, \{\tilde{x}_i\}_{i \in \X})$ be data for $C$.  Then $(\{y_i\}_{i \in \X}, \{\tilde{y}_i\}_{i \in \X})$ is data for $C$ if and only if the following are equal and independent of $i$ for $i \in \X$:
\[
    x_i - y_i, \qquad \tilde{x}_i - \tilde{y}_i.
\]
\end{lemma}
\begin{proof}
Routine.
\end{proof}

Let $C \in \text{Mat}_{\X}(\K)$ be Cauchy. For $i\in \X$, let $\mu_i$ denote the $i$th column sum of $C^{-1}$ and let $\tilde{\mu}_i$ denote the $i$th row sum of $C^{-1}$. Let $J \in \text{Mat}_\X(\K)$ denote the all-ones matrix and consider the matrices $C^{-1}J$ and $JC^{-1}$.  Observe that, for each $l \in \X$, $\mu_i = (JC^{-1})_{li}$ and $\tilde{\mu}_i = (C^{-1}J)_{il}$.
Let $(\{x_i\}_{i \in \X}, \{\tilde{x}_i\}_{i \in \X})$ be data for $C$. Using the fact that $C(C^{-1}J) = J$ and $(JC^{-1})C = J$, it is easily verified that $\{\mu_i\}_{i \in \X}$ and $\{\tilde{\mu}_i\}_{i \in \X}$ satisfy the following equations for each $j \in \X$:
\begin{equation}\label{eq:sums}
    \sum_{i \in \X} \frac{\mu_i}{x_i - \tilde{x}_j} = 1, \qquad \sum_{i \in \X} \frac{-\tilde{\mu}_i}{\tilde{x}_i - x_j} = 1.
\end{equation}
\begin{corollary}\label{cor:CinvJ}
Let $C \in \rm{Mat}_\X(\K)$ be Cauchy and let $(\{x_i\}_{i \in \X}, \{\tilde{x}_i\}_{i \in \X})$ be data for $C$.  Then for $i, j \in \X$, the following {\rm(i),(ii)} hold.
\begin{enumerate}
 \item[{\rm (i)}] The $i$th column sum of $C^{-1}$ is equal to $\dfrac{\prod_{k\in\X}(x_i - \tilde{x}_k)}{\prod_{k \in \X\setminus i} (x_i - x_k)}$.
 \item[{\rm (ii)}]The $i$th row sum of $C^{-1}$ is equal to $- \dfrac{\prod_{k\in\X}(\tilde{x}_i - x_k)}{\prod_{k \in \X\setminus i} (\tilde{x}_i - \tilde{x}_k)}$.
\end{enumerate}
\end{corollary}
\begin{proof}
By Lemma \ref{lem:lambdas} and (\ref{eq:sums}).
\end{proof}
\section{Cauchy pairs}\label{sec:cp}
For the rest of the paper, fix a vector space $V$ over $\K$ with dimension $|\X|$.  Let End($V$) denote the $\K$-algebra consisting of all $\K$-linear transformations from $V$ to $V$.  For $S \in \text{End}(V)$ and $W \subseteq V$, we call $W$ an {\it eigenspace} of $S$ whenever $W \ne 0$ and there exists $\theta \in \K$ such that $W = \{v \in V | Sv = \theta v\}$.  In this case, $\theta$ is called the {\it eigenvalue} of $S$ corresponding to $W$.  We say that $S$ is {\it diagonalizable} whenever $V$ is spanned by the eigenspaces of $S$.  We say that $S$ is {\it multiplicity-free} whenever $S$ is diagonalizable and each eigenspace of $S$ has dimension 1.

\begin{definition}\label{def:cp}
\rm
By a {\it Cauchy pair} on $V$, we mean an ordered pair $(X,\tilde{X})$ of elements in End($V$) that satisfy the following three conditions.
\begin{enumerate}
 \item Each of $X, \tilde{X}$ is diagonalizable.
 \item $X-\tilde{X}$ has rank 1.
 \item There does not exist a subspace $W$ of $V$ such that $XW \subseteq W, \tilde{X} W \subseteq W, W \ne 0, W \ne V$.
\end{enumerate}

We call $V$ the {\it underlying vector space} for $(X, \tilde{X})$.  We say that $(X, \tilde{X})$ is {\it over} $\mathbb{K}$.
\end{definition}

\noindent The following result is immediate from Definition \ref{def:cp}.
\begin{lemma}\label{lem:dual}
If $(X,\tilde{X})$ is a Cauchy pair on $V$, then $(\tilde{X},X)$ is a Cauchy pair on $V$.
\end{lemma}

For the rest of the paper, we adopt the following notational convention.

\begin{definition}\label{def:tilde}
\rm
 For a Cauchy pair $(X, \tilde{X})$ on $V$ and for any object $f$ that we associate with $(X, \tilde{X})$, let $\tilde{f}$ denote the corresponding object for the Cauchy pair $(\tilde{X}, X)$.
\end{definition}

For a Cauchy pair $(X, \tilde{X})$ on $V$, we will need the following facts about the eigenvalues and eigenspaces of $X$ and $\tilde{X}$.

\begin{lemma}
Let $(X, \tilde{X})$ denote a Cauchy pair on $V$.  Then the eigenvalues of $X$ and $\tilde{X}$ are contained in $\K$.
\end{lemma}
\begin{proof}
Follows from Definition \ref{def:cp}(i).
\end{proof}

\begin{theorem}\label{thm:commoneval}
Let $(X, \tilde{X})$ denote a Cauchy pair on $V$.  Then there does not exist a scalar in $\K$ that is both an eigenvalue for $X$ and an eigenvalue for $\tilde{X}$.
\end{theorem}
\begin{proof}

Suppose that $\theta \in \K$ is an eigenvalue for $X$.  We show that $\theta$ is not an eigenvalue for $\tilde{X}$. If $\text{dim}\, V = 1$, then $X$ acts as $\theta I$ on $V$ and $\tilde{X}$ acts as $\mu I$ on $V$ for some $\mu \in \K$.  By Definition \ref{def:cp}(ii), $X-\tilde{X} \ne 0$, so $\mu \ne \theta$. Thus, $\theta$ is not an eigenvalue for $\tilde{X}$.

\medskip
We now assume that $\text{dim}\  V > 1$.  Let $W = (\tilde{X}-\theta I)V$.  To show that $\theta$ is not an eigenvalue for $\tilde{X}$, we show that $W = V$.  By Definition \ref{def:cp}(iii), it suffices to show that $W \ne 0,\  \tilde{X}W\subseteq W,$ and $X W \subseteq W$.

\medskip
 First, we show that $W \ne 0$.  Suppose that $W = 0$.  Then $\tilde{X}$ acts as $\theta I$ on $V$, so any subspace of $V$ is $\tilde{X}$-invariant.  Thus, the span of any eigenvector for $X$ is a proper nonzero subspace of $V$ that is both $X$- and $\tilde{X}$-invariant, contradicting Definition \ref{def:cp}(iii).  Therefore $W \ne 0$.

\medskip
Next, we show that $\tilde{X} W \subseteq W$.  Observe that
\[
\tilde{X}W = \tilde{X} (\tilde{X} - \theta I)V = (\tilde{X} - \theta I)\tilde{X}V \subseteq (\tilde{X} - \theta I)V = W.
\]
Therefore $\tilde{X}W \subseteq W$.

\medskip
Finally, we show that $X W \subseteq W$.  Since $\tilde{X} W \subseteq W$, it suffices to show that $(\tilde{X}-X)W \subseteq W$.  Let $v \in V$ be an eigenvector for $X$ with eigenvalue $\theta$.  Then $\theta v = Xv$, so $(\tilde{X}-\theta I)v = (\tilde{X}-X)v \in (\tilde{X}-X)V$.  By the definition of $W$, $(\tilde{X}-\theta I)v \in W$.  Also, $(\tilde{X} - \theta I)v \ne 0$ since $X$ and $\tilde{X}$ have no common eigenvectors by Definition \ref{def:cp}(iii).  Thus $(\tilde{X}-\theta I)v$ is a nonzero element of $W \cap (\tilde{X}-X)V$.  Since $\tilde{X}-X$ has rank 1, this implies that $(\tilde{X}-X)V \subseteq W$.  Therefore $(\tilde{X}-X)W \subseteq W$.
\end{proof}

\begin{theorem}\label{thm:multfree} Let $(X,\tilde{X})$ denote a Cauchy pair on $V$.  Then each of $X, \tilde{X}$ is  multiplicity-free.
\end{theorem}
\begin{proof}
By Lemma \ref{lem:dual}, it suffices to show that $X$ is multiplicity-free.

\medspace
By Definition \ref{def:cp}(i), $X$ is diagonalizable.  Let $W$ denote an eigenspace of $X$.  We show that $W$ has dimension 1. Let $\theta$ be the eigenvalue of $X$ corresponding to $W$. By Theorem \ref{thm:commoneval}, $\theta$ is not an eigenvalue for $\tilde{X}$, so $\tilde{X} - \theta I$ is invertible on $V$.  Thus, the dimension of $W$ is equal to the dimension of $(\tilde{X} - \theta I)W$.  Since $W \ne 0$, it follows that $(\tilde{X} - \theta I)W \ne 0$.  Observe that $X$ acts as $\theta I$ on $W$, so
\begin{equation}\label{eq:containment}
(\tilde{X} - \theta I)W = (\tilde{X} - X)W \subseteq (\tilde{X}-X)V.
\end{equation}
By Definition \ref{def:cp}(ii), $\tilde{X} - X$ has rank 1, so $(\tilde{X} - X)V$ has dimension 1.  Thus, equality holds in (\ref{eq:containment}) and $(\tilde{X}-\theta I)W$ has dimension 1.  Therefore $W$ has dimension 1.
\end{proof}

For a Cauchy pair $(X, \tilde{X})$ on $V$, Theorem \ref{thm:multfree} shows that we can index the eigenvalues of $X$ and the eigenvalues of $\tilde{X}$ using the elements of $\X$.  The indexing is arbitrary, but it will be useful to fix an indexing in order to talk about specific eigenvalues of $X$ and $\tilde{X}$ and objects related to these eigenvalues.
\begin{definition}\label{def:evaldata}\rm
Let $(X, \tilde{X})$ denote a Cauchy pair on $V$.  Let $\{x_i\}_{i \in \X}$ (resp. $\{\tilde{x}_i\}_{i \in \X}$) denote the eigenvalues of $X$ (resp. $\tilde{X})$. We say that $(\{x_i\}_{i \in \X}, \{\tilde{x}_i\}_{i \in \X})$ is {\it eigenvalue data} for $(X, \tilde{X})$.
\end{definition}

\section{Isomorphism and equivalence of Cauchy pairs}\label{sec:iso}

In this section, we introduce the notions of isomorphism, affine isomorphism, and equivalence for Cauchy pairs.

\begin{definition}\label{def:iso}
\rm
Let $(X, \tilde{X})$ and $(Y, \tilde{Y})$ denote Cauchy pairs over $\K$.  By an {\it isomorphism of Cauchy pairs from $(X, \tilde{X})$ to $(Y, \tilde{Y})$}, we mean an isomorphism $\phi$ of $\K$-vector spaces from the vector space underlying $(X, \tilde{X})$ to the vector space underlying $(Y, \tilde{Y})$ such that $\phi X = Y \phi$ and $\phi \tilde{X} = \tilde{Y} \phi$.  The Cauchy pairs $(X, \tilde{X})$ and $(Y, \tilde{Y})$ are said to be {\it isomorphic} whenever there exists an isomorphism of Cauchy pairs from $(X, \tilde{X})$ to $(Y, \tilde{Y})$.
\end{definition}

The next result enables us to construct a family of Cauchy pairs from a given Cauchy pair.
\begin{lemma}\label{lem:affinetrans}
Let $(X, \tilde{X})$ denote a Cauchy pair on $V$ with eigenvalue data $(\{x_i\}_{i \in \X}, \{\tilde{x}_i\}_{i \in \X})$.  Let $\xi, \zeta$ denote scalars from $\K$ with $\xi\ne 0$.  Then
\begin{equation}\label{affinetrans}
(\xi X + \zeta I, \xi \tilde{X} + \zeta I)
\end{equation}
is a Cauchy pair on $V$ with eigenvalue data $(\{\xi x_i + \zeta\}_{i \in \X}, \{\xi \tilde{x}_i + \zeta\}_{i \in \X})$.
\end{lemma}

\begin{definition}
\rm
Referring to Lemma \ref{lem:affinetrans}, we call (\ref{affinetrans}) the {\it affine transformation of $(X, \tilde{X})$ with parameters $\xi, \zeta$}.
\end{definition}

\begin{definition}\label{def:affineiso}
\rm
Let $(X, \tilde{X})$ and $(Y, \tilde{Y})$ denote Cauchy pairs over $\K$.  We say $(X, \tilde{X})$ and $(Y, \tilde{Y})$ are {\it affine isomorphic} whenever $(X, \tilde{X})$ is isomorphic to an affine transformation of $(Y, \tilde{Y})$.
\end{definition}

We mention a special case of affine isomorphism for Cauchy pairs.
\begin{definition}\label{def:equivalence}
\rm
Let $(X, \tilde{X})$ and $(Y, \tilde{Y})$ denote Cauchy pairs over $\K$.  We say that $(X, \tilde{X})$ and $(Y, \tilde{Y})$ are ${\it equivalent}$ whenever there exists $\zeta \in \K$ such that $(X, \tilde{X})$ is isomorphic to the affine transformation of $(Y, \tilde{Y})$  with parameters $1, \zeta$.
\end{definition}

\section{Associated Cauchy matrices and Cauchy pairs}

In this section, we associate Cauchy matrices to Cauchy pairs.  We show that every Cauchy matrix is associated to some Cauchy pair in this way.

\begin{proposition}\label{prop:assoc}
Let $(X, \tilde{X})$ denote a Cauchy pair on $V$ and let $(\{x_i\}_{i \in \X}, \{\tilde{x}_i\}_{i \in \X})$ denote eigenvalue data for $(X,\tilde{X})$. Then there exists a unique Cauchy matrix $C \in {\rm Mat}_\X(\K)$ such that $(\{x_i\}_{i \in \X}, \{\tilde{x}_i\}_{i \in \X})$ is data for $C$.
\end{proposition}
\begin{proof}
By Theorem \ref{thm:commoneval}, $x_i - \tilde{x_j} \ne 0$ for all $i, j \in \X$.  Define $C \in \text{Mat}_{\X}(\K)$ to have $(i,j)$-entry $(x_i - \tilde{x_j})^{-1}$ for $i, j \in \X$.  By Theorem \ref{thm:multfree}, $\{x_i\}_{i \in \X}$ and $\{\tilde{x}_i\}_{i \in \X}$ are mutually distinct.  Thus, by Definition \ref{def:cm} and Definition \ref{def:data}, $C$ is Cauchy with data $(\{x_i\}_{i \in \X}, \{\tilde{x}_i\}_{i \in \X})$.
\end{proof}

\begin{definition}\label{def:assoc}
\rm
Let $(X, \tilde{X})$ denote a Cauchy pair on $V$ and let $C \in \text{Mat}_\X(\K)$ denote a Cauchy matrix. With reference to Proposition \ref{prop:assoc}, we say that $C$ is {\it associated} to $(X, \tilde{X})$ whenever there exists eigenvalue data $(\{x_i\}_{i \in \X}, \{\tilde{x}_i\}_{i \in \X})$ for $(X, \tilde{X})$ such that $(\{x_i\}_{i \in \X}, \{\tilde{x}_i\}_{i \in \X})$  is data for $C$.
\end{definition}

For $P \in \text{Mat}_\X(\K)$, we say that $P$ is a {\it permutation matrix} whenever $P$ has exactly one entry 1 in each row and column and all other entries 0. For $M, N \in \text{Mat}_\X(\K)$, $M$ and $N$ are called {\it permutation equivalent} whenever there exist permutation matrices $P, Q \in \text{Mat}_\X(\K)$ such that $M = PNQ$.  In other words, $M$ and $N$ are permutation equivalent whenever $M$ can be obtained by permuting the rows and columns of $N$.

\begin{lemma}\label{lem:permequiv}
Let $(X, \tilde{X})$ denote a Cauchy pair on $V$ and let $C \in \text{Mat}_{\X}(\K)$ be a Cauchy matrix associated to $(X, \tilde{X})$. Let $C^\prime \in \text{Mat}_{\X}(\K)$ be Cauchy. Then $C^\prime$ is associated to $(X, \tilde{X})$ if and only if $C$ and $C^\prime$ are permutation equivalent.
\end{lemma}
\begin{proof}
Follows from Definition \ref{def:assoc} and the comments above Definition \ref{def:evaldata}.
\end{proof}

Next, we show that every Cauchy matrix is associated to a Cauchy pair.  We first recall some concepts from linear algebra.

\medskip
Let $S \in \text{End}(V)$ and let $\{v_i\}_{i \in \X}$ denote a basis for $V$.  For $M \in \text{Mat}_{\X}(\K)$, we say that $M$ {\it represents} $S$ with respect to $\{v_i\}_{i \in \X}$ whenever $S v_j = \sum_{i \in \X} M_{ij}v_i$ for all $j \in \X$.  Let $\{w_i\}_{i \in \X}$ denote a basis for $V$.  By the {\it transition matrix} from $\{v_i\}_{i \in \X}$ to $\{w_i\}_{i \in \X}$, we mean the matrix $B \in \text{Mat}_{\X}(\K)$ such that $w_j = \sum_{i \in \X} B_{ij}v_i$ for all $j \in \X$.  In this case, $B^{-1}$ exists and is equal to the transition matrix from $\{w_i\}_{i \in \X}$ to $\{v_i\}_{i \in \X}$.

\medskip
Let $C \in \text{Mat}_{\X}(\K)$ be Cauchy.  We now construct a Cauchy pair $(X, \tilde{X})$ on $V$ to which $C$ is associated.

\begin{notation}\label{notation}
\rm
Let $C \in \text{Mat}_{\X}(\K)$ be Cauchy and let $(\{x_i\}_{i \in \X}, \{\tilde{x}_i\}_{i \in \X})$ be data for $C$.  Let $D, \tilde{D}  \in \text{Mat}_{\X}(\K)$ be diagonal with entries $D_{ii} = x_i$ and $\tilde{D}_{ii} = x_i$ for $i \in \X$.  Let $\{v_i\}_{i \in \X}$ and $\{w_i\}_{i \in \X}$ denote bases for $V$ such that $C$ is the transition matrix from $\{v_i\}_{i \in \X}$ to $\{w_i\}_{i \in \X}$.  Let $X \in {\rm End}(V)$ be the linear transformation represented by $D$ with respect to $\{v_i\}_{i \in \X}$. Let $\tilde{X}\in {\rm End}(V)$ be the linear transformation represented by $\tilde{D}$ with respect to $\{w_i\}_{i \in \X}$.
\end{notation}

The matrices $C, D, \tilde{D}$ are related in the following way.
\begin{proposition}\label{prop:displ}
With reference to Notation \ref{notation},
\begin{equation}\label{eq:displ}
    D C - C \tilde{D} = J.
\end{equation}
\end{proposition}
\begin{proof}
Routine by matrix multiplication.
\end{proof}

\begin{theorem}\label{thm:cpfromcm}
With reference to Notation \ref{notation}, $(X, \tilde{X})$ is a Cauchy pair on $V$.  Moreover, $C$ is associated to $(X, \tilde{X})$.
\end{theorem}
\begin{proof}
To show that $(X,\tilde{X})$ is a Cauchy pair, we show that $(X, \tilde{X})$ satisfies conditions (i)--(iii) of Definition \ref{def:cp}.
By the construction in Notation \ref{notation}, (i) is satisfied. Next, we show that (ii) holds.  With respect to $\{w_i\}_{i \in \X}$, $X$ is represented by the matrix $C^{-1} D C$.  By Proposition \ref{prop:displ}, $C^{-1} D C = \tilde{D} + C^{-1}J$.  Thus, with respect to $\{w_i\}_{i \in \X}$, $X-\tilde{X}$ is represented by the matrix $\tilde{D} + C^{-1}J - \tilde{D} = C^{-1}J$.  This matrix has rank 1, so $X-\tilde{X}$ has rank 1.  Therefore (ii) holds.

\medskip
Finally, we show that (iii) holds.  Let $0 \ne W \subseteq V$ such that $XW \subseteq W$ and $\tilde{X}W \subseteq W$.  We show that $W = V$.  Since $\tilde{X}$ is multiplicity-free, any $\tilde{X}$-invariant subspace of $V$ is the span of some set of eigenvectors for $\tilde{X}$.  Thus, there exists $\mathfrak{I} \subseteq \X$ such that $W = \text{Span}\{w_i \ |\  i \in \mathfrak{I}\}$.  To show that $W = V$, it suffices to show that $\mathfrak{I} = \X$.  Since $W$ is invariant under $X$ and $\tilde{X}$, $W$ is also invariant under $X - \tilde{X}$.  Therefore $(X-\tilde{X})w_j \in W$ for all $j \in \mathfrak{I}$.  As noted above, $X-\tilde{X}$ is represented by $C^{-1}J$ with respect to $\{w_i\}_{i \in \X}$.  Thus, for $j \in \mathfrak{I}$, $(X-\tilde{X}) w_j = \sum_{i\in \X} (C^{-1}J)_{ij}w_i$.  Observe that $(C^{-1}J)_{ij}$ is the $i$th row sum of $C^{-1}$. By Corollary \ref{cor:CinvJ}, $(C^{-1}J)_{ij} \ne 0$ for $i \in \X$, so $(X-\tilde{X})w_j$ is not contained in the span of any proper subset of $\{w_i\}_{i \in \X}$.  Therefore $\mathfrak{I} = \X$, so (iii) holds.

\medskip
The final assertion follows by construction.
\end{proof}

Observe that, with reference to Notation \ref{notation}, $C$ acts as a transition matrix from an $X$-eigenbasis for $V$ to an $\tilde{X}$-eigenbasis for $V$. In Section \ref{sec:mainresults}, we will show that every Cauchy matrix associated to a Cauchy pair acts in this way. That is, we will show that if a Cauchy matrix $C \in \text{Mat}_\X(\K)$ is associated to a Cauchy pair $(X, \tilde{X})$ on $V$, then there exists an $X$-eigenbasis $\{v_i\}_{i \in \X}$ for $V$ and an $\tilde{X}$-eigenbasis $\{w_i\}_{i \in \X}$ for $V$ such that $C$ is the transition matrix from $\{v_i\}_{i \in \X}$ to $\{w_i\}_{i \in \X}$.

\medskip
We also note the following fact about equivalent Cauchy pairs.
\begin{lemma}\label{lem:equivtoassoc}
Let $(X, \tilde{X})$ denote a Cauchy pair on $V$ and let $C \in {\rm Mat}_{\X}(\K)$ be a Cauchy matrix associated to $(X, \tilde{X})$. Let $(Y, \tilde{Y})$ be a Cauchy pair and suppose that $(Y, \tilde{Y})$ is equivalent to $(X, \tilde{X})$.  Then $C$ is associated to $(Y, \tilde{Y})$.
\end{lemma}
\begin{proof}
By Definition \ref{def:equivalence}, there exists $ \zeta \in \K$ such that $(Y, \tilde{Y})$ is isomorphic to $(X+\zeta I, \tilde{X}+\zeta I)$. By Definition \ref{def:assoc}, there exists eigenvalue data $(\{x_i\}_{i \in \X}, \{\tilde{x}_i\}_{i \in \X})$ for $(X, \tilde{X})$ such that $(\{x_i\}_{i \in \X}, \{\tilde{x}_i\}_{i \in \X})$ is data for $C$. By Lemma \ref{lem:affinetrans}, $(Y, \tilde{Y})$ has eigenvalue data $(\{x_i + \zeta\}_{i \in \X}, \{\tilde{x}_i + \zeta\})$.  By Lemma \ref{lem:equivdata}, $(\{x_i + \zeta\}_{i \in \X}, \{\tilde{x}_i + \zeta\})$ is also data for $C$.  Therefore $C$ is associated to $(Y, \tilde{Y})$.
\end{proof}
We just proved Lemma \ref{lem:equivtoassoc}. We will show in Section \ref{sec:mainresults} that the converse of Lemma \ref{lem:equivtoassoc} holds as well.  That is, we will show that if a Cauchy matrix is associated to Cauchy pairs $(X, \tilde{X})$ and $(Y, \tilde{Y})$, then $(X, \tilde{X})$ and $(Y, \tilde{Y})$ are equivalent.

\section{Standard Bases}\label{sec:standardbases}

In this section, we introduce the notion of a standard basis for the vector space underlying a Cauchy pair.

\medskip
Until future notice, let $(X, \tilde{X})$ denote a Cauchy pair on $V$ and  let $(\{x_i\}_{i \in \X}, \{\tilde{x}_i\}_{i \in \X})$ denote eigenvalue data for $(X, \tilde{X})$. Observe that
\begin{equation}\label{eq:traces}
\text{tr}(X) = \sum_{i \in \X} x_i, \qquad \text{tr}(\tilde{X}) = \sum_{i \in \X} \tilde{x}_i,
\end{equation}
where tr denotes trace.

\medskip
We now recall the notion of a primitive idempotent. Let $S$ denote a multiplicity-free element of $\text{End}(V)$ with eigenvalues $\{\theta_i\}_{i \in \X}$.  For $i \in \X$, let $V_i$ denote the eigenspace of $S$ corresponding to $\theta_i$.  Define $E_i \in \text{End}(V)$ such that $(E_i - I)V_i = 0$ and $E_i V_j = 0$ for $j \in \X, j \ne i$. Observe that (i) $V_i = E_i V \ (i \in \X)$, (ii) $E_i E_j = \delta_{ij}E_i \ (i,j \in \X)$, (iii)$S E_i = \theta_i E_i = E_i S \ (i \in \X)$,  (iv) $I = \sum_{i \in \X} E_i$, and (v) $S = \sum_{i \in \X}\theta_i E_i$.

\medskip
By linear algebra,
\[
    E_i = \prod_{\substack{j \in \X\\ j\ne i}} \frac{S - \theta_j I}{\theta_i - \theta_j}  \qquad (i \in \X).
\]
We call $E_i$ the {\it primitive idempotent} of $S$ corresponding to $\theta_i$.

\medskip
For $i \in \X$, let $E_i$ denote the primitive idempotent of $X$ corresponding to $x_i$ and let $\tilde{E}_i$ denote the primitive idempotent of $\tilde{X}$ corresponding to $\tilde{x}_i$.
\begin{definition}\label{def:delta}
\rm
 Define $\Delta, \tilde{\Delta} \in \text{End}(V)$ by $\Delta = X - \tilde{X}$ and $\tilde{\Delta} = \tilde{X} - X = -\Delta$.  Observe that $\Delta V = \tilde{\Delta}V$. By Definition \ref{def:cp}, this common subspace has dimension 1.
\end{definition}

By (\ref{eq:traces}),
\begin{equation}\label{eq:deltatrace}
\text{tr}(\Delta) = \sum_{i \in \X} (x_i - \tilde{x}_i).
\end{equation}

\begin{lemma}\label{lem:eta}
Let $\eta$ denote a nonzero vector in $\Delta V$.  Then for $i \in \X$,  $E_i \eta$ is nonzero and therefore a basis for $E_i V$.  Moreover, $\{E_i \eta\}_{i \in \X}$ is a basis for $V$.
\end{lemma}
\begin{proof}
Suppose that there exists $i \in \X$ such that $E_i \eta = 0$.  Let $W = \sum_{j \in \X\setminus i} E_j V$.

Recall that $I = \sum_{j \in \X} E_j$.  Applying each side of this equation to $\eta$, we find
\[
\eta = \sum_{j \in \X} E_j \eta = \sum_{j \in \X \setminus  i} E_j \eta \in W.
\]

Thus $W\ne 0$.  Note that $X W \subseteq W$ since $X E_j= x_j E_j$ for $j \in \X$.  Also, $\eta$ spans $\Delta V$, so
\[
 \Delta W \subseteq \Delta V = \text{Span}\{\eta\} \subseteq W.
\]
Then by Definition \ref{def:delta}, $\tilde{X} W \subseteq W$.  By construction, $W\ne V$, so this contradicts Definition 2.1(iii).  Therefore $E_i \eta \ne 0$ for all $i \in \X$.

\medskip
The last assertion follows from this and the fact that $V = \sum_{i \in \X} E_i V$ (direct sum).
\end{proof}

\begin{definition}\label{def:Xstd}
\rm
Let $\{v_i\}_{i \in \X}$ denote a basis for $V$.  We call $\{v_i\}_{i \in \X}$ {\it X-standard} whenever there exists $0 \ne \eta \in \Delta V$ such that $v_i = E_i \eta$ for all $i \in \X$.
\end{definition}

\begin{lemma}\label{lem:Xstd}
Let $\{v_i\}_{i \in \X}$ denote vectors in $V$ that are not all zero.  Then $\{v_i\}_{i \in \X}$ is an $X$-standard basis for $V$ if and only the following {\rm (i), (ii)} hold.
\begin{enumerate}
\item[{\rm (i)}] $v_i \in E_i V$ for all $i \in \X$;
\item[{\rm (ii)}] $\displaystyle\sum_{j \in \X} v_j \in \Delta V$.
\end{enumerate}
\end{lemma}
\begin{proof}
First, assume that $\{v_i\}_{i \in \X}$ is an $X$-standard basis for $V$.  By Definition \ref{def:Xstd}, there exists $0 \ne \eta \in \Delta V$ such that $v_i = E_i \eta$ for all $i \in \X$.  Thus $\{v_i\}_{i \in \X}$ satisfy (i).  Recall that $I = \sum_{j\in\X} E_j$.  Applying each side of this equation to $\eta$, we find $\eta = \sum_{j \in \X} E_j \eta = \sum_{j\in\X} v_j$.  Thus $\{v_i\}_{i\in\X}$ satisfy (ii).

\medskip
Conversely, assume that $\{v_i\}_{i \in \X}$ satisfy (i), (ii) above.  Define $\eta = \sum_{j \in \X} v_j$.  By (ii), $\eta \in \Delta V$.  Using (i), we find that $E_i \eta = \sum_{j \in \X} \delta_{ij}v_j = v_i$ for $i \in \X$. Observe that $\eta \ne 0$ since at least one of $\{v_i\}_{i \in \X}$ is nonzero.  By Lemma \ref{lem:eta},  $\{v_i\}_{i \in \X}$ is a basis for $V$.  Thus $\{v_i\}_{i \in \X}$ is an $X$-standard basis for $V$ by Definition \ref{def:Xstd}.
\end{proof}

\begin{lemma}\label{lem:Xstdunique}
Let $\{v_i\}_{i \in \X}$ be an $X$-standard basis for $V$.  Let $\{v_i^\prime\}_{i \in \X}$ denote vectors in $V$.  Then $\{v_i^\prime\}_{i \in \X}$ is an $X$-standard basis for $V$ if and only if there exists $0 \ne \theta \in \K$ such that $v_i^\prime = \theta v_i$ for $i \in \X$.
\end{lemma}
\begin{proof}
Follows from Definition \ref{def:Xstd} and the fact that $\Delta V$ has dimension 1.
\end{proof}

\section{The scalars $\alpha_i$}\label{sec:alphas}
We continue to study the Cauchy pair $(X, \tilde{X})$ on $V$ with eigenvalue data $(\{x_i\}_{i\in\X}, \{\tilde{x}_i\}_{i\in\X})$.  Recall the map $\Delta \in \text{End}(V)$ from Definition \ref{def:delta}. In this section, we introduce some scalars to help describe the space $\Delta V$.
\begin{definition}\label{def:alphas}
\rm
For $i \in \X$, define
\[
 \alpha_i = \text{tr}(E_i\Delta).
\]
\end{definition}

\begin{proposition}\label{prop:trace}
With reference to Definition \ref{def:alphas},
\[
    \sum_{i \in \X} \alpha_i = {\rm tr}(\Delta).
\]
\end{proposition}
\begin{proof}
 Recall that $I = \sum_{i \in \X} E_i$. Thus $\Delta = \sum_{i \in \X} E_i \Delta$.  Take the trace of both sides and apply Definition \ref{def:alphas}. The result follows.
\end{proof}

We ultimately want to express $\alpha_i$ in terms of $\{x_i\}_{i\in\X}, \{\tilde{x}_i\}_{i \in \X}$.  To this end, we rewrite Proposition \ref{prop:trace} in the following way.
\begin{corollary}\label{cor:alphasum}
With reference to Definition \ref{def:alphas},
\[\sum_{i \in \X} \alpha_i = \sum_{i \in \X} (x_i - \tilde{x}_i).\]
\end{corollary}
\begin{proof}
Follows from Proposition \ref{prop:trace} and (\ref{eq:deltatrace}).
\end{proof}

\begin{lemma}\label{lem:eideltaei}
For all $i \in \X$,
\begin{equation}\label{eq:eideltaei}
E_i \Delta E_i = \alpha_i E_i.
\end{equation}
\end{lemma}
\begin{proof}
Let $\mathcal{A} = \text{End}(V)$.  Since $E_i$ is a primitive idempotent, it follows from linear algebra that the space $E_i \mathcal{A} E_i$ is spanned by $E_i$.  Thus, there exists $\theta_i \in \K$ such that
\begin{equation}\label{eq:theta}
E_i \Delta E_i = \theta_i E_i.
\end{equation}
To compute $\theta_i$, we take the trace of both sides of (\ref{eq:theta}).  The trace of the left side is equal to  $ \text{tr}(E_i \Delta)$, which is equal to $\alpha_i$ by Definition \ref{def:alphas}.  The trace of the right side is equal to $\theta_i$ since $\text{tr}(E_i) = 1$.  Therefore $\theta_i = \alpha_i$.
\end{proof}

\begin{lemma}\label{lem:alphaXstd}
Let $\{v_i\}_{i \in \X}$ be an $X$-standard basis for $V$ and let $\eta = \sum_{i \in \X} v_i$.  Then for $i \in \X$,
\[
\Delta v_i = \alpha_i \eta.
\]
\end{lemma}
\begin{proof}
By Lemma \ref{lem:Xstd} and Definition \ref{def:cp}(ii), $\eta$ is a basis for $\Delta V$.  Fix $i \in \X$.  Then $\Delta v_i \in \Delta V$, so there exists $\theta_i \in \K$ such that $\Delta v_i = \theta_i \eta$.  Applying $E_i$ to both sides of this equations gives
\begin{equation}\label{eq:theta2}
E_i \Delta v_i = \theta_i E_i \eta.
\end{equation}
To compute $\theta_i$, evaluate the left side of $(\ref{eq:theta2})$ using Lemma \ref{lem:eideltaei} and the fact that $v_i = E_i \eta$.  The result follows.
\end{proof}

\begin{corollary}\label{cor:nonzero}
For all $i \in \X$, $\alpha_i$ is nonzero.
\end{corollary}
\begin{proof}
Suppose there exists $i \in \X$ such that $\alpha_i = 0$.  Let $\{v_i\}_{i \in \X}$ be an $X$-standard basis for $V$.  By Lemma \ref{lem:alphaXstd}, $\Delta v_i = 0$.  By Definition \ref{def:Xstd}, $X v_i = x_i v_i$.  Therefore $\text{Span}\{v_i\}$ is $\Delta$-invariant and $X$-invariant, so it is also $\tilde{X}$-invariant. Since $\text{Span}\{v_i\} \ne 0$,  it follows that $\text{Span}\{v_i\}  = V$ by Definition \ref{def:cp}(iii).  Then $\Delta V = 0$, contradicting Definition \ref{def:cp}(ii).  Therefore $\alpha_i \ne 0$ for all $i \in \X$.
\end{proof}

\begin{definition}\label{def:natural}
\rm
Let $\{v_i\}_{i \in \X}$ be an $X$-standard basis for $V$.  For $A \in \text{End}(V)$, let $A^\natural$ denote the matrix in $\text{Mat}_\X(\K)$ that represents $A$ with respect to $\{v_i\}_{i \in \X}$.  This defines a $\K$-algebra isomorphism $\natural : \text{End}(V) \to \text{Mat}_\X(\K)$ that sends $A \mapsto A^\natural$.
\end{definition}

It follows from Lemma \ref{lem:Xstdunique} that $\natural$ is independent of the choice of $X$-standard basis for $(X, \tilde{X})$.

\begin{lemma}\label{lem:matrep}
For the map $\natural$ from Definition \ref{def:natural}, the following {\rm (i)--(iii)} hold.
\begin{enumerate}
\item[{\rm (i)}] $X^\natural$ is diagonal with $(i,i)$-entry $x_i$ for $i \in \X$.
\item[{\rm (ii)}] $\Delta^\natural$ has $(i,j)$-entry $\alpha_j$ for $i, j \in \X$.
\item[{\rm (iii)}] $\tilde{X}^\natural$ has $(i,j)$-entry $x_i \delta_{ij} - \alpha_j$ for $i,j \in \X$.
\end{enumerate}
\end{lemma}
\begin{proof}
(i) Let $\{v_i\}_{i\in\X}$ be an $X$-standard basis for $V$.  Then $Xv_i = x_i v_i$ by Definition \ref{def:Xstd} and the comments above Definition \ref{def:delta}.  The result follows from Definition \ref{def:natural}.

(ii) By Lemma \ref{lem:alphaXstd}.

(iii)  Follows from (i), (ii) using $\tilde{X} = X - \Delta$.
\end{proof}

\section{Transition matrices}
We continue to study the Cauchy pair $(X, \tilde{X})$ on $V$ with eigenvalue data $(\{x_i\}_{i\in\X}, \{\tilde{x}_i\}_{i\in\X})$. In Sections \ref{sec:standardbases} and \ref{sec:alphas}, we introduced the notion of an $X$-standard basis for $V$ and we the defined scalars $\{\alpha_i\}_{i \in \X}$.  Invoking Definition \ref{def:tilde}, we also have the notion of an $\tilde{X}$-standard basis for $V$ and the scalars $\{\tilde{\alpha}_i\}_{i \in \X}$.  In this section, we compute the transition matrices between an $X$-standard basis and an $\tilde{X}$-standard basis.

\begin{definition}\label{defn:index}
\rm
Let $\{v_i\}_{i \in \X}$ be an $X$-standard basis for $V$ and let $\{w_i\}_{i \in \X}$ be an $\tilde{X}$-standard basis for $V$.  By Lemma \ref{lem:Xstd}, $\sum_{i \in \X} v_i \in \Delta V$ and $\sum_{i \in \X} w_i \in \tilde{\Delta} V = \Delta V$. Therefore there exists $0 \ne \gamma \in \K$ such that
\begin{equation}\label{eq:index}
    \sum_{i \in \X} w_i = \gamma \sum_{i \in \X} v_i.
\end{equation}
We call $\gamma$ the {\it index} of $(X, \tilde{X})$ corresponding to $\{v_i\}_{i \in \X}, \{w_i\}_{i \in \X}$.
\end{definition}

The index $\gamma$ is free in the following sense.
\begin{lemma}\label{lem:freeindex}
Let $\{v_i\}_{i \in \X}$ be an $X$-standard basis for $V$ and let $0 \ne \gamma \in \K$.  Then there exists a unique $\tilde{X}$-standard basis $\{w_i\}_{i \in \X}$ for $V$ such that $\gamma$ is the index of $(X, \tilde{X})$ corresponding to $\{v_i\}_{i \in \X}, \{w_i\}_{i \in \X}$.
\end{lemma}
\begin{proof}
Let $\eta = \sum_{i \in \X} v_i$.  Observe that $\eta \ne 0$ since $\{v_i\}_{i \in \X}$ are linearly independent. By Lemma \ref{lem:Xstd}(ii), $\eta \in \Delta V$.  Set $w_i = \gamma \tilde{E}_i \eta$ for $i \in \X$.  By applying Lemma \ref{lem:Xstd} to the Cauchy pair $(\tilde{X}, X)$, we see that $\{w_i\}_{i \in \X}$ is an $\tilde{X}$-standard basis for $V$.  Furthermore, $\sum_{i \in \X} w_i = \gamma \eta$.  Therefore $\gamma$ is the index of $(X, \tilde{X})$ corresponding to $\{v_i\}_{i \in \X}, \{w_i\}_{i \in \X}$.

\medskip
The uniqueness assertion follows from Lemma \ref{lem:Xstdunique}.
\end{proof}

Until further notice, we assume the following setup.
\begin{notation}\label{not:standard}
\rm
Fix an $X$-standard basis $\{\epsilon_i\}_{i \in \X}$ for $V$ and an $\tilde{X}$-standard basis $\{\tilde{\epsilon}_i\}_{i \in \X}$ for $V$.  Let $\gamma$ denote the index of $(X, \tilde{X})$ corresponding to $\{\epsilon_i\}_{i\in\X}, \{\tilde{\epsilon}_i\}_{i \in \X}$.  In view of Definition \ref{def:tilde}, let $\tilde{\gamma}$ denote the index of $(\tilde{X}, X)$ corresponding to $\{\tilde{\epsilon}_i\}_{i \in \X}, \{\epsilon_i\}_{i \in \X}$.  Let $T \in \text{Mat}_\X(\K)$ denote the transition matrix from $\{\epsilon_i\}_{i\in\X}$ to $\{\tilde{\epsilon}_i\}_{i\in\X}$.  Let $\tilde{T} \in \text{Mat}_\X(\K)$ denote the transition matrix from $\{\tilde{\epsilon}_i\}_{i \in \X}$ to $\{\epsilon_i\}_{i \in \X}$.
\end{notation}
Observe that $\tilde{\gamma} = \gamma^{-1}$, and that $T$, $\tilde{T}$ are inverses.

\begin{theorem}\label{thm:transmat}
With reference to Notation \ref{not:standard}, the following hold for $i, j \in \X$,
\begin{enumerate}
\item[{\rm (i)}] $T_{ij} = - \gamma\tilde{\alpha}_j / (x_i - \tilde{x}_j)$;
\item[{\rm (ii)}] $\tilde{T}_{ij} = - \gamma^{-1}\alpha_j/(\tilde{x}_i - x_j)$.
\end{enumerate}
\end{theorem}
\begin{proof}

(i)  We evaluate $\Delta \tilde{\epsilon}_j$ in two ways.

First, recall that $\Delta = -\tilde{\Delta}$ by Definition \ref{def:delta}.  Use Lemma \ref{lem:alphaXstd} applied to $(\tilde{X}, X)$ and Definition \ref{defn:index} to show that
\begin{equation}\label{eq:way1}
\Delta \tilde{\epsilon}_j = -\gamma \tilde{\alpha}_j \sum_{i \in \X} \epsilon_i.
\end{equation}

Second, we have $\tilde{\epsilon}_j = \sum_{i \in \X} T_{ij} \epsilon_i$. Apply $X$ to both sides to get $X \tilde{\epsilon}_j = \sum_{i \in \X} T_{ij} x_i \epsilon_i$.  Since $\{\tilde{\epsilon}_i\}_{i \in \X}$ is $\tilde{X}$-standard, $\tilde{X} \tilde{\epsilon}_j = \tilde{x}_j \tilde{\epsilon}_j$.  By Definition \ref{def:delta},
\begin{equation}\label{eq:way2}
\Delta \tilde{\epsilon}_j = \left(\sum_{i \in \X} T_{ij} x_i \epsilon_i\right)  - \tilde{x}_j \tilde{\epsilon}_j = \sum_{i \in \X} (x_i - \tilde{x}_j) T_{ij} \epsilon_i.
\end{equation}

Equate the coefficients of $\epsilon_i$ in (\ref{eq:way1}) and (\ref{eq:way2}) and solve for $T_{ij}$ to get (i).

(ii) Apply (i) to the Cauchy pair $(\tilde{X}, X)$.
\end{proof}

\section{Identities involving $\alpha_i, \tilde{\alpha}_i$}\label{section:identities}
We continue to study the Cauchy pair $(X, \tilde{X})$ on $V$ with eigenvalue data $(\{x_i\}_{i\in\X}, \{\tilde{x}_i\}_{i\in\X})$. Recall the matrices $T$ and $\tilde{T}$ from Notation \ref{not:standard}.  In this section, we use $T$ and $\tilde{T}$ to establish a few identities involving the scalars $\{\alpha_i\}_{i \in \X}, \{\tilde{\alpha}_i\}_{i \in \X}$.
\begin{lemma}\label{lem:sumiden}
The following {\rm (i), (ii)} hold for all $j \in \X$.
\begin{enumerate}
\item[{\rm (i)}] $\displaystyle\sum_{i \in \X} \frac{\alpha_i}{x_i - \tilde{x}_j} = 1$;
\item[{\rm (ii)}] $\displaystyle\sum_{i \in \X} \frac{\tilde{\alpha}_i}{\tilde{x}_i - x_j} = 1$.
\end{enumerate}
\end{lemma}
\begin{proof}
We first prove (ii). Recall the map $\natural$ from Definition \ref{def:natural}.  With reference to Definition \ref{def:tilde}, we also have a map $\tilde{\natural}$.  To simplify notation, we use $\sharp$ to denote $\tilde{\natural}$.    Recall $T$ from Notation \ref{not:standard}. By linear algebra,
\begin{equation}\label{eq:XnatT}
X^\natural T = T X^\sharp.
\end{equation}
For $k \in \X$, we compute the $(j,k)$-entry of each side of (\ref{eq:XnatT}).  First, we compute the $(j,k)$-entry of $X^\natural T$.  By Lemma \ref{lem:matrep}, Theorem \ref{thm:transmat}, and matrix multiplication,
\begin{equation}\label{eq:MxT}
    (X^\natural T)_{jk} =  - \gamma \frac{x_j \tilde{\alpha}_k}{x_j - \tilde{x}_k}.
\end{equation}
Next, we compute the $(j,k)$-entry of $T X^\sharp$.  By Lemma \ref{lem:matrep} applied to $(\tilde{X}, X)$, Theorem \ref{thm:transmat}, and matrix multiplication,
\begin{equation}\label{eq:TMx}
    (T X^\sharp)_{jk} = -\gamma \left( \frac{\tilde{x}_k \tilde{\alpha}_k}{x_j - \tilde{x}_k} - \tilde{\alpha}_k \sum_{i \in \X} \frac{\tilde{\alpha}_i}{x_j - \tilde{x_i}}\right).\\
\end{equation}
Equate (\ref{eq:MxT}) and (\ref{eq:TMx}) and simplify using Corollary \ref{cor:nonzero} to get (ii).

\medspace
To prove (i), apply (ii) to the Cauchy pair $(\tilde{X}, X)$.
\end{proof}

Recall the linear system (\ref{eq:system}).  By Lemma \ref{lem:sumiden}(i), $\{\alpha_i\}_{i\in\X}$ is a solution to (\ref{eq:system}) when $a_i = x_i$ and $b_i = \tilde{x}_i$ for all $i \in \X$.  By Lemma \ref{lem:sumiden}(ii), $\{\tilde{\alpha}_i\}_{i\in\X}$ is a solution to (\ref{eq:system}) when $a_i = \tilde{x}_i$ and $b_i = x_i$ for all $i \in \X$. In view of Lemma \ref{lem:lambdas}, we obtain the following result.
\begin{corollary}\label{cor:alphaxi}
For all $i \in \X$,
\begin{enumerate}
\item[{\rm (i)}] $\alpha_i = \displaystyle\frac{\prod_{k \in \X} (x_i - \tilde{x}_k)}{\prod_{k \in \X \setminus i} (x_i - x_k)}$;
\item[{\rm (ii)}] $\tilde{\alpha}_i = \displaystyle\frac{\prod_{k \in \X} (\tilde{x}_i - x_k)}{\prod_{k \in \X \setminus i} (\tilde{x}_i - \tilde{x}_k)}$.
\end{enumerate}
\end{corollary}
\begin{proof}
Follows from Lemma \ref{lem:sumiden} and Lemma \ref{lem:lambdas}.
\end{proof}

\begin{corollary}
Let $C \in {\rm Mat}_\X(\K)$ be Cauchy with data $(\{x_i\}_{i \in \X}, \{\tilde{x}_i\}_{i \in \X})$. Then for all $i \in \X$, $\alpha_i$ is equal to the $i$th column sum of $C^{-1}$ and $-\tilde{\alpha}_i$ is equal to $i$th row sum of $C^{-1}$.
\end{corollary}
\begin{proof}
Follows from Corollary \ref{cor:CinvJ} and Corollary \ref{cor:alphaxi}.
\end{proof}

We mention another identity involving $\{\alpha_i\}_{i\in\X}, \{\tilde{\alpha}_i\}_{i \in \X}$.
\begin{lemma}\label{lem:iden2}
For all $i \in \X$,
\begin{equation}\label{eq:iden2}
 \sum_{j \in \X} \frac{\alpha_i \tilde{\alpha_j}}{(x_i-\tilde{x}_j)^2} = -1.
\end{equation}
\end{lemma}
\begin{proof}
Recall that $T \tilde{T} = I$.  In this equation, compute the $(i,i)$-entry of each side using Theorem \ref{thm:transmat} and matrix multiplication.  The result follows.
\end{proof}

\section{A Bilinear Form}

In this section, we associate a certain bilinear form to each Cauchy pair.

\medskip
By a {\it bilinear form on V}, we mean a map $\langle \ , \ \rangle : V \times V \to \K$ that satisfies the following four conditions for all $u, v, w \in V$ and for all $\theta \in \K$: (i) $\langle u+v,w \rangle = \langle u,w\rangle + \langle v, w \rangle$; (ii) $\langle \theta u, v \rangle = \theta \langle u, v \rangle$; (iii) $\langle u, v+w \rangle = \langle u, v \rangle + \langle u, w\rangle$; (iv) $\langle u, \theta v\rangle = \theta \langle u,v \rangle$. We observe that a scalar multiple of a bilinear form on $V$ is a bilinear form on $V$.  A bilinear form $\langle \ , \ \rangle$ on $V$ is said to be {\it symmetric} whenever $\langle u, v \rangle = \langle v, u\rangle$ for all $u, v \in V$.  Let $\langle \ , \ \rangle$ denote a bilinear form on $V$.  Then the following are equivalent: (i) there exists a nonzero $u \in V$ such that $\langle u, v \rangle = 0$ for all $v \in V$; (ii) there exists a nonzero $v \in V$ such that $\langle u, v\rangle = 0$ for all $u \in V$.  The form $\langle \ , \ \rangle$ is said to be {\it degenerate} whenever (i), (ii) hold and {\it nondegenerate} otherwise.  Given a bilinear form $\langle \ , \ \rangle$ on $V$, we abbreviate $||u||^2 = \langle u, u \rangle$ for all $u \in V$.

\medskip
We continue to study the Cauchy pair $(X, \tilde{X})$ on $V$ with eigenvalue data $(\{x_i\}_{i\in\X}, \{\tilde{x}_i\}_{i\in\X})$.  We will show that there exists a nonzero bilinear form $\langle \ , \ \rangle$ on $V$ such that
\begin{equation}\label{eqs:invform}
    \langle X u, v \rangle = \langle u, X v\rangle, \ \   \langle \tilde{X} u, v \rangle = \langle u, \tilde{X} v \rangle \qquad (u, v \in V).
\end{equation}
Before showing that $\langle \ , \ \rangle$ exists, we investigate the consequences of (\ref{eqs:invform}).

\begin{definition}\rm
By an {\it $(X, \tilde{X})$-invariant form} on $V$, we mean a bilinear form on $V$ that satisfies (\ref{eqs:invform}).
\end{definition}
\begin{lemma}\label{lem:orthogonal}
Let $\langle \ , \ \rangle$ be an $(X, \tilde{X})$-invariant form on $V$.  Then $\{E_i V\}_{i \in \X}$ are mutually orthogonal with respect to $\langle \ , \ \rangle$ and $\{\tilde{E}_i V\}_{i \in \X}$ are mutually orthogonal with respect to $\langle \ , \ \rangle$.
\end{lemma}
\begin{proof}
Concerning the first claim, pick distinct $i, j \in \X$.  Let $u \in E_i V$ and let $v \in E_j V$.  We show $\langle u, v \rangle = 0$. By (\ref{eqs:invform}),
\[
 x_i \langle u, v \rangle = \langle X u, v \rangle = \langle u, X v \rangle = x_j \langle u, v \rangle.
\]
We have $x_i \ne x_j$ since $i \ne j$.  Therefore $\langle u, v \rangle = 0$.

\medskip
The second claim is similarly shown.
\end{proof}

\begin{lemma}\label{lem:symmetric}
Let $\langle \ , \ \rangle$ be a nonzero $(X, \tilde{X})$-invariant form on $V$.   Then $\langle \ , \ \rangle$ is symmetric and nondegenerate.
\end{lemma}
\begin{proof}
First, we show that $\langle \ , \ \rangle$ is symmetric. For $i \in \X$, let $0 \ne u_i \in E_i V$ and note that $\{u_i\}_{i \in \X}$ is a basis for $V$.  It suffices to show that $\langle u_i , u_j \rangle = \langle u_j, u_i \rangle$ for all $i, j\in \X$.   The case when $i = j$ is clear.  For $i \ne j$, $\langle u_i, u_j\rangle = 0$ by Lemma \ref{lem:orthogonal}.  Thus $\langle \ , \ \rangle$ is symmetric.

\medskip
Next, we show that $\langle \ , \ \rangle$ is nondegenerate.  Let $V^\perp = \{v \in V :  \langle v, w \rangle = 0 \ \ \forall w \in V\}$.  We show that $V^\perp = 0$.    To do this, we show that $V^\perp \ne V, X V^\perp \subseteq V^\perp$, and $\tilde{X}V^\perp \subseteq V^\perp$.  Since $\langle \ ,\ \rangle$ is nonzero, $V^\perp \ne V$.  Let $v \in V^\perp$.  By (\ref{eqs:invform}), for all $w \in V$, $\langle X v, w \rangle = \langle v, X w\rangle = 0$ and $\langle \tilde{X} v, w \rangle = \langle v, \tilde{X} w\rangle = 0$.  Therefore $X w, \tilde{X}w \in V^\perp$, so $X V^\perp \subseteq V^\perp$ and $\tilde{X} V^\perp \subseteq V^\perp$.  By Definition \ref{def:cp}(iii), $V^\perp = 0$.
\end{proof}

\begin{lemma}\label{lem:nonzeronorms}
Let $\langle \ , \ \rangle$ be a nonzero $(X, \tilde{X})$-invariant form on $V$.  Let $\{u_i\}_{i \in \X}$ denote an $X$-standard basis for $V$.  Then $||u_i||^2 \ne 0$ for all $i \in \X$.
\end{lemma}
\begin{proof}
Suppose there exists $i \in \X$ such that $||u_i||^2 = 0$.  Then by Lemma \ref{lem:orthogonal}, $\langle u_i, u_j\rangle = 0$ for all $j \in \X$. Since $\{u_i\}_{i \in \X}$ is a basis for $V$, $\langle u_i, u \rangle = 0$ for all $u \in V$.  Thus $\langle \ , \ \rangle$ is degenerate, contradicting Lemma \ref{lem:symmetric}.  Therefore $||u_i||^2 \ne 0$ for all $i \in \X$.
\end{proof}

\begin{theorem}\label{thm:formexistance}
There exists a nonzero $(X, \tilde{X})$-invariant form $\langle \ , \ \rangle$ on $V$.  Moreover, this form is unique up to multiplication by a nonzero scalar in $\K$.
\end{theorem}
\begin{proof}
Let $\{u_i\}_{i \in \X}$ be an $X$-standard basis for $V$.  Recall the scalars $\{\alpha_i\}_{i \in \X}$ from Definition \ref{def:alphas}. Define a bilinear form $\langle \ , \ \rangle$ on $V$ such that $\langle u_i, u_j\rangle = \delta_{ij}\alpha_i$ for $i, j \in \X$.  We show that $\langle \ , \ \rangle$ satisfies (\ref{eqs:invform}).

\medskip
Concerning the equation on the left in (\ref{eqs:invform}), with reference to Definition \ref{def:Xstd},
\[
\langle X u_i, u_j \rangle = x_i \alpha_i \delta_{ij}  = \langle u_i, X u_j \rangle \qquad (i, j \in \X).
\]

Concerning the equation on the right in (\ref{eqs:invform}), by Definition \ref{def:delta}, it suffices to show that  $\langle \Delta u_i, u_j \rangle = \langle u_i, \Delta u_j \rangle$ for all $i, j \in \X$.
By Lemma \ref{lem:alphaXstd},
\[
\left\langle \Delta u_i, u_j \right\rangle =  \alpha_i \alpha_j = \left\langle u_i, \Delta u_j \right\rangle \qquad (i, j \in \X).
\]
We have now shown that $\langle \ ,\  \rangle$ satisfies (\ref{eqs:invform}), so $\langle \ ,\ \rangle$ is $(X, \tilde{X})$-invariant.

\medskip
Next, we show that $\langle \ , \ \rangle$ is unique up to multiplication by a nonzero scalar in $\K$.  Suppose that $\langle \ , \ \rangle^\prime$ is a nonzero  $(X, \tilde{X})$-invariant form on $V$.  We show that there exists $0 \ne c \in \K$ such that $\langle \ , \ \rangle = c \langle \ , \ \rangle^\prime$.

\medskip
Pick $i \in \X$.  By Lemma \ref{lem:nonzeronorms}, $\langle u_i , u_i \rangle \ne 0$ and $\langle u_i, u_i\rangle^\prime \ne 0$.  Thus, there exists $0 \ne c \in \K$ such that $\langle u_i, u_i\rangle = c \langle u_i, u_i\rangle^\prime$.  Define a bilinear form $( \ , \ )$ on $V$ such that $( u , v ) =  \langle u , v \rangle - c \langle u, v \rangle^\prime$ for all $u, v \in V$.  It is easily verified that $( \ , \ )$ satisfies (\ref{eqs:invform}). By construction, $(u_i, u) = 0$ for all $u \in V$, so $(\ , \ )$ is degenerate. Therefore $(\ , \ ) = 0$ by Lemma \ref{lem:symmetric}.  Consequently, $\langle \ , \ \rangle = c \langle \ , \ \rangle^\prime$.
\end{proof}

Until further notice, we fix a nonzero $(X, \tilde{X})$-invariant form $\langle \ , \ \rangle$ on $V$.

\medskip
 Recall the $X$-standard basis $\{\epsilon_i\}_{i\in \X}$, the $\tilde{X}$-standard basis $\{\tilde{\epsilon}_i\}_{i \in \X}$, and the index $\gamma$ from Notation \ref{not:standard}. Recall the scalars $\{\alpha_i\}_{i \in \X}, \{\tilde{\alpha}_i\}_{i \in \X}$ from Definition \ref{def:alphas}.
\begin{lemma}\label{lem:norms}
The following {\rm (i), (ii)} hold.
\begin{enumerate}
\item[{\rm (i)}] $||\epsilon_i||^2/ \alpha_i$ is independent of $i$ for $i \in \X$.
\item[{\rm (ii)}] $||\tilde{\epsilon}_i||^2 / \tilde{\alpha}_i$ is independent of $i$ for $i \in \X$.
\end{enumerate}
We let $\rho$ denote the common value of $||\epsilon_i||^2/\alpha_i$ and let $\tilde{\rho}$ denote the common value of $||\tilde{\epsilon}_i||^2/\tilde{\alpha}_i$.
\end{lemma}
\begin{proof}
(i) Let $i, j \in \X$ with $i \ne j$.  We show that $||\epsilon_i||^2/\alpha_i = ||\epsilon_j||^2/\alpha_j$.  By (\ref{eqs:invform}), $\langle \tilde{X} \epsilon_i, \epsilon_j \rangle = \langle \epsilon_i, \tilde{X} \epsilon_j\rangle$.  By Lemma \ref{lem:matrep}, $\tilde{X}\epsilon_i = x_i \epsilon_i - \alpha_i \sum_{k \in \X} \epsilon_k$.  Using this and Lemma \ref{lem:orthogonal}, $\langle \tilde{X} \epsilon_i, \epsilon_j \rangle = -\alpha_i ||\epsilon_j||^2$.  Similarly, $\langle \epsilon_i, \tilde{X} \epsilon_j\rangle = - \alpha_j ||\epsilon_i||^2$.  The result follows.

\medskip
(ii) Similar.
\end{proof}

\begin{lemma}\label{lem:epsilons}
With reference to Lemma \ref{lem:norms}, $\tilde{\rho} = -\rho \gamma^2 $.
\end{lemma}
\begin{proof}
Fix $j \in \X$. We evaluate $||\tilde{\epsilon}_j||^2$ in two ways.  By Lemma \ref{lem:norms}, $||\tilde{\epsilon}_j||^2 = \tilde{\rho}\tilde{\alpha}_j$.

\medskip
Recall the matrix $T$ from Notation \ref{not:standard}. Observe that $\tilde{\epsilon}_j = \sum_{i \in \X} T_{ij} \epsilon_i$, so
\begin{align*}
||\tilde{\epsilon}_j||^2 &= \sum_{i \in \X} T_{ij}^2 ||\epsilon_i||^2 &\text{by Lemma \ref{lem:orthogonal}}\\
				&= \sum_{i \in \X} \left(\frac{-\gamma \tilde{\alpha}_j}{x_i-\tilde{x}_j}\right)^2 ||\epsilon_i||^2 &\text{by Theorem \ref{thm:transmat}}\\
				&= \rho \gamma^2 \tilde{\alpha}_j \sum_{i \in \X} \frac{\alpha_i \tilde{\alpha}_j}{(x_i - \tilde{x}_j)^2} &\text{by Lemma \ref{lem:norms}}\\
				&= -\rho \gamma^2 \tilde{\alpha}_j &\text{by Lemma \ref{lem:iden2}}.
\end{align*}
The result follows since $\tilde{\alpha}_j \ne 0$ by Corollary \ref{cor:nonzero}.

\end{proof}

\begin{lemma}\label{lem:innerprods}
The following {\rm (i)--(iii)} hold for $i, j \in \X$.
\begin{enumerate}
\item[{\rm(i)}] $\langle \epsilon_i, \tilde{\epsilon}_j \rangle / ||\epsilon_i||^2 = -\gamma\tilde{\alpha}_j / (x_i - \tilde{x}_j)$;
\item[{\rm(ii)}] $\langle \epsilon_i, \tilde{\epsilon}_j \rangle / ||\tilde{\epsilon}_j||^2 = \gamma^{-1}\alpha_i / (x_i - \tilde{x}_j)$;
\item[{\rm(iii)}] $\langle \epsilon_i, \tilde{\epsilon}_j\rangle^2 / ||\epsilon_i||^2 ||\tilde{\epsilon}_j||^2 = -\alpha_i \tilde{\alpha}_j / (x_i-\tilde{x}_j)^2$.
\end{enumerate}
\end{lemma}
\begin{proof}
(i)  Recall the matrix $T$ from Notation \ref{not:standard}.  Observe that $\tilde{\epsilon}_j = \sum_{i \in \X} T_{ij} \epsilon_i$, so $\langle \epsilon_i, \tilde{\epsilon}_j\rangle = T_{ij} ||\epsilon_i||^2$ by Lemma \ref{lem:orthogonal}.  The result follows from this and Theorem \ref{thm:transmat}(i).

(ii) Recall the matrix $\tilde{T}$ from Notation \ref{not:standard}.  Observe that $\epsilon_i = \sum_{j \in \X}\tilde{T}_{ji}\tilde{\epsilon}_j$, so $\langle \epsilon_i, \tilde{\epsilon}_j \rangle = \tilde{T}_{ji} ||\tilde{\epsilon}_j||^2$ by Lemma \ref{lem:orthogonal}.  The result follows from this and Theorem \ref{thm:transmat}(ii).

(iii) Follows from (i) and (ii).
\end{proof}

\noindent Let $\{u_i\}_{i \in \X}$ and $\{u_i^\prime\}_{i \in \X}$ denote bases for $V$.  We say that $\{u_i\}_{i \in \X}$ and $\{u_i^\prime\}_{i \in \X}$ are {\it dual} with respect to $\langle \ , \ \rangle$ whenever $\langle u_i, u_j^\prime\rangle = \delta_{ij}$ for $i,j \in \X$.  Given a basis $\{u_i\}_{i \in \X}$ for $V$, there exists a unique basis for $V$ that is dual to $\{u_i\}_{i \in \X}$ with respect to $\langle \ , \ \rangle$.  We denote this basis by $\{u_i^\ast\}_{i \in \X}$.

\begin{notation}\label{not:duals}\rm
Let $\{\epsilon_i^\ast\}_{i \in \X}$ denote the basis for $V$ that is dual to $\{\epsilon_i\}_{i \in \X}$ with respect to $\langle \ , \ \rangle$.  Let $\{\tilde{\epsilon}_i^\ast\}_{i \in \X}$ denote the basis for $V$ that is dual to $\{\tilde{\epsilon}_i\}_{i \in \X}$ with respect to $\langle \ , \ \rangle$.
\end{notation}

\begin{lemma}\label{lem:duals}
For $i \in \X$, the following {\rm (i), (ii)} hold.
\begin{enumerate}
\item[{\rm (i)}] $\epsilon_i^\ast = \epsilon_i/||\epsilon_i||^2$;
\item[{\rm (ii)}] $\tilde{\epsilon}_i^\ast = \tilde{\epsilon}_i/||\tilde{\epsilon}_i||^2$.
\end{enumerate}
\end{lemma}
\begin{proof}
Follows from Notation \ref{not:duals} by routine linear algebra.
\end{proof}

\begin{lemma}\label{lem:duals2} For $i \in \X$, the following {\rm (i), (ii)} hold.
\begin{enumerate}
\item[{\rm (i)}] $||\epsilon_i^\ast||^2  = \rho^{-1} \alpha_i^{-1}$;
\item[{\rm (ii)}] $||\tilde{\epsilon}_i^\ast||^2 = -\rho^{-1}\gamma^{-2}\tilde{\alpha}_i^{-1}$.
\end{enumerate}
\end{lemma}
\begin{proof}
Use Lemma \ref{lem:duals} and Lemma \ref{lem:norms}.
\end{proof}

\begin{lemma}\label{lem:dualinnerprods}
For $i, j \in \X$, the following {\rm (i)--(v)} hold.
\begin{enumerate}
\item[{\rm (i)}] $\langle \epsilon_i^\ast, \tilde{\epsilon}_j \rangle = -\gamma\tilde{\alpha}_j / (x_i - \tilde{x}_j)$;
\item[{\rm (ii)}] $\langle \epsilon_i, \tilde{\epsilon}_j^\ast \rangle = \gamma^{-1}\alpha_i / (x_i - \tilde{x}_j)$;
\item[{\rm (iii)}] $\langle \epsilon_i^\ast, \tilde{\epsilon}_j^\ast \rangle / ||\epsilon_i^\ast||^2 = \gamma^{-1}\alpha_i / (x_i - \tilde{x}_j)$;
\item[{\rm (iv)}] $\langle \epsilon_i^\ast, \tilde{\epsilon}_j^\ast \rangle / ||\tilde{\epsilon}_j^\ast||^2 = -\gamma\tilde{\alpha}_j / (x_i - \tilde{x}_j)$;
\item[{\rm (v)}] $\langle \epsilon_i^\ast, \tilde{\epsilon}_j^\ast\rangle^2 / ||\epsilon_i^\ast||^2 ||\tilde{\epsilon}_j^\ast||^2 = -\alpha_i \tilde{\alpha}_j / (x_i-\tilde{x}_j)^2$.
\end{enumerate}
\end{lemma}
\begin{proof}
(i), (ii) Use Lemma \ref{lem:duals} and Lemma \ref{lem:innerprods}.

(iii) Follows from Lemma \ref{lem:duals} and (ii).

(iv) Follows from Lemma \ref{lem:duals} and (i).

(v) Follows from (iii), (iv).
\end{proof}

\section{More Transition Matrices}

We continue to study the Cauchy pair $(X, \tilde{X})$ on $V$ with eigenvalue data $(\{x_i\}_{i\in\X}, \{\tilde{x}_i\}_{i\in\X})$. From Notation \ref{not:standard} and Notation \ref{not:duals}, we now have four bases for $V$:
\begin{equation}\label{eq:bases}
\{\epsilon_i\}_{i \in \X}, \ \{\tilde{\epsilon}_i\}_{i \in \X}, \ \{\epsilon_i^\ast\}_{i \in \X}, \ \{\tilde{\epsilon}_i^\ast\}_{i \in \X}
\end{equation}

In Theorem \ref{thm:transmat}, we gave the transition matrix from $\{\epsilon_i\}_{i \in \X}$ to $\{\tilde{\epsilon}_i\}_{i \in \X}$ and the transition matrix from $\{\tilde{\epsilon}_i\}_{i \in \X}$ to $\{\epsilon_i\}_{i \in \X}$.  In this section, we give the transition matrices for the remaining pairs of bases from $(\ref{eq:bases})$.

\medskip
We recall a property of transition matrices.  Let $\{u_i\}_{i \in \X}$ and $\{v_i\}_{i \in \X}$ denote bases for $V$.  Let $B \in \text{Mat}_\X(\K)$ be the transition matrix from $\{u_i\}_{i \in \X}$ to $\{v_i\}_{i \in \X}$.  If $\{u_i\}_{i \in \X}$, $\{v_i\}_{i \in \X}$ are each mutually orthogonal with respect to a bilinear form $\langle \ , \ \rangle$, then the entries of $B$ are given by

\begin{equation}\label{eq:entries}
B_{ij} = \frac{\langle v_j, u_i\rangle}{||u_i||^2} \qquad (i, j \in \X).
\end{equation}

Recall the scalars $\{\alpha_i\}_{i \in \X}, \{\tilde{\alpha}_i\}_{i \in \X}$ from Definition \ref{def:alphas}, the scalar $\gamma$ from Notation \ref{not:standard}, and the scalar $\rho$ from Lemma \ref{lem:norms}.
\begin{lemma}\label{lem:dualtransmat}
With reference to {\rm(\ref{eq:bases})}, the following {\rm (i),(ii)} hold.
\begin{enumerate}
\item[{\rm (i)}] The transition matrix from $\{\epsilon_i\}_{i \in \X}$ to $\{\epsilon_i^\ast\}_{i \in \X}$ is diagonal with $(i,i)$-entry
\[
	\rho^{-1}\alpha_i^{-1} \qquad (i \in \X).
\]
\item[{\rm (ii)}] The transition matrix from $\{\tilde{\epsilon}_i\}_{i \in \X}$ to $\{\tilde{\epsilon}_i^\ast\}_{i \in \X}$ is diagonal with $(i,i)$-entry
\[
	-\rho^{-1}\gamma^{-2}\tilde{\alpha}_i^{-1} \qquad (i \in \X).
\]
\end{enumerate}
\end{lemma}
\begin{proof}
(i) By Lemma \ref{lem:duals} and Lemma \ref{lem:norms}.

(ii) By Lemma \ref{lem:duals}, Lemma \ref{lem:norms}, and Lemma \ref{lem:epsilons}.
\end{proof}

\begin{lemma}\label{lem:transmatbig}
With reference to {\rm(\ref{eq:bases})}, the following {\rm (i)--(iii)} hold.
\begin{enumerate}
\item[{\rm (i)}] The transition matrix from $\{\epsilon_i\}_{i \in \X}$ to $\{\tilde{\epsilon}_i^\ast\}_{i \in \X}$ has $(i,j)$-entry
\[
    \rho^{-1}\gamma^{-1}\frac{1}{x_i - \tilde{x}_j} \qquad (i,j \in \X),
\]
and the transition matrix from the $\{\tilde{\epsilon}_i^\ast\}_{i \in \X}$ to $\{\epsilon_i\}_{i \in \X}$ has $(i,j)$-entry
\[
    \rho\gamma\frac{\tilde{\alpha}_i \alpha_j}{\tilde{x}_i - x_j} \qquad (i,j \in \X).
\]
\item[{\rm (ii)}] The transition matrix from $\{\epsilon_i^\ast\}_{i \in \X}$ to $\{\tilde{\epsilon}_i\}_{i \in \X}$ has $(i,j)$-entry
\[
    -\rho\gamma\frac{\alpha_i \tilde{\alpha}_j}{x_i - \tilde{x}_j} \qquad (i,j \in \X),
\]
and the transition matrix from $\{\tilde{\epsilon}_i\}_{i \in \X}$ to $\{\epsilon_i^\ast\}_{i \in \X}$ has $(i,j)$-entry
\[
    -\rho^{-1}\gamma^{-1}\frac{1}{\tilde{x}_i - x_j} \qquad (i,j \in \X).
\]
\item[{\rm (iii)}] The transition matrix from $\{\epsilon_i^\ast\}_{i \in \X}$ to $\{\tilde{\epsilon}_i^\ast\}_{i \in \X}$ has $(i,j)$-entry
\[
    \gamma^{-1}\frac{\alpha_i}{x_i - \tilde{x}_j} \qquad (i,j \in \X),
\]
and the transition matrix from $\{\tilde{\epsilon}_i^\ast\}_{i \in \X}$ to $\{\epsilon_i^\ast\}_{i \in \X}$ has $(i,j)$-entry
\[
    \gamma\frac{\tilde{\alpha}_i}{\tilde{x}_i - x_j} \qquad (i,j \in \X).
\]
\end{enumerate}
\end{lemma}
\begin{proof}
For each transition matrix, compute (\ref{eq:entries}) using Lemma \ref{lem:dualinnerprods}, Lemma \ref{lem:norms}, Lemma \ref{lem:epsilons}, and Lemma \ref{lem:duals2}, as well as the fact that $\langle \ , \ \rangle$ is symmetric by Lemma \ref{lem:symmetric}.
\end{proof}

\section{Summary using matrices}
We continue to study the Cauchy pair $(X, \tilde{X})$ on $V$ with eigenvalue data $(\{x_i\}_{i\in\X}, \{\tilde{x}_i\}_{i\in\X})$. We now summarize our results on inner products and transition matrices in terms of matrices.

\begin{notation}\label{not:matrices} \rm
Let $A, \tilde{A} \in \text{Mat}_\X(\K)$ denote diagonal matrices with $A_{ii} = \alpha_i$ and $\tilde{A}_{ii} = \tilde{\alpha}_i$ for $i \in \X$. Let $C \in \text{Mat}_\X(\K)$ denote the Cauchy matrix with data $(\{x_i\}_{i \in \X}, \{\tilde{x}_i\}_{i \in \X})$ and let $\tilde{C} \in \text{Mat}_\X(\K)$ denote the Cauchy matrix with data $(\{\tilde{x}_i\}_{i \in \X},\{x_i\}_{i \in \X})$.
\end{notation}
By construction, $C$ is associated to the Cauchy pair $(X, \tilde{X})$ and $\tilde{C}$ is associated to the Cauchy pair $(\tilde{X}, X)$ by Definition \ref{def:assoc}. Note that $C^\top = -\tilde{C}$. Moreover, the inverses of $C, \tilde{C}$ can be expressed in the following way.
\begin{proposition} With reference to Notation \ref{not:matrices}, $C^{-1} = \tilde{A}CA$ and $\tilde{C}^{-1} = A \tilde {C} \tilde{A}$.\end{proposition}
\begin{proof}
By Lemma \ref{lem:transmatbig}(i), $\rho^{-1}\gamma^{-1} C$ is the transition matrix from $\{\epsilon_{i}\}_{i \in \X}$ to $\{\tilde{\epsilon}_i^\ast\}_{i \in \X}$ and $\rho \gamma \tilde{A}CA$ is the transition matrix from $\{\tilde{\epsilon}_i^\ast\}_{i \in \X}$ to $\{\epsilon_i\}_{i \in \X}$.  Thus $\rho^{-1}\gamma^{-1} C$ and $\rho \gamma \tilde{A}CA$ are inverses, so  $C$ and $\tilde{A}CA$ are inverses. Therefore $C^{-1} = \tilde{A}CA$.

\medspace
Using Lemma \ref{lem:transmatbig}(ii), a similar argument shows the second assertion.
\end{proof}

\begin{lemma}
In the diagram below, we display inner product matrices between the four bases in {\rm(\ref{eq:bases})} using Notation \ref{not:matrices}:

\begin{equation*}
\xymatrix@C=16pc@R=12pc{
\{\epsilon_i\}_{i \in \X}
        \ar@(ur,ul)[]_{\rho A}
        \ar@<.5ex>[r]^{-\rho\gamma A C \tilde{A}}
        \ar@<.5ex>[dr]^(.75){\gamma^{-1}AC}
        \ar@{-}[d]_{I}
    & \{\tilde{\epsilon}_i\}_{i \in \X}
        \ar@(ur,ul)[]_{-\rho \gamma^2  \tilde{A}}
        \ar@<.5ex>[l]^{\rho\gamma\tilde{A}\tilde{C} A}
        \ar@<.5ex>[dl]^(.75){\gamma \tilde{A}\tilde{C}}|\hole
        \ar@{-}[d]^{I}\\
    \{\epsilon_i^\ast\}_{i \in \X}
        \ar@(dr,dl)[]^{\rho^{-1}A^{-1}}
        \ar@<.5ex>[ur]^(.75){-\gamma C \tilde{A}}|\hole
        \ar@<.5ex>[r]^{\rho^{-1}\gamma^{-1} C}
    & \{\tilde{\epsilon}_i^\ast\}_{i \in \X}
        \ar@(dr,dl)[]^{-{\rho}^{-1}\gamma^{-2}\tilde{A}^{-1}}
        \ar@<.5ex>[l]^{-\rho^{-1}\gamma^{-1}\tilde{C}}
        \ar@<.5ex>[ul]^(.75){-\gamma^{-1} \tilde{C} A}
}
\end{equation*}
where $\{u_i\}_{i \in \X} \displaystyle\mathrel{\mathop{\longrightarrow}^{\mathrm{B}}}\{v_i\}_{i \in \X}$ means $B_{ij} = \langle u_i, v_j\rangle$ for all $i, j \in \X$.  The direction arrow is omitted when $B$ is symmetric.
\end{lemma}
\begin{proof}
Use Notation \ref{not:matrices} and verify using Lemma \ref{lem:norms}, Lemma \ref{lem:epsilons}, Lemma \ref{lem:innerprods}, Lemma \ref{lem:duals2}, and Lemma \ref{lem:dualinnerprods}.
\end{proof}

\begin{lemma} In the diagram below, we display the transition matrices between the four bases in {\rm(\ref{eq:bases})} using Notation \ref{not:matrices}:
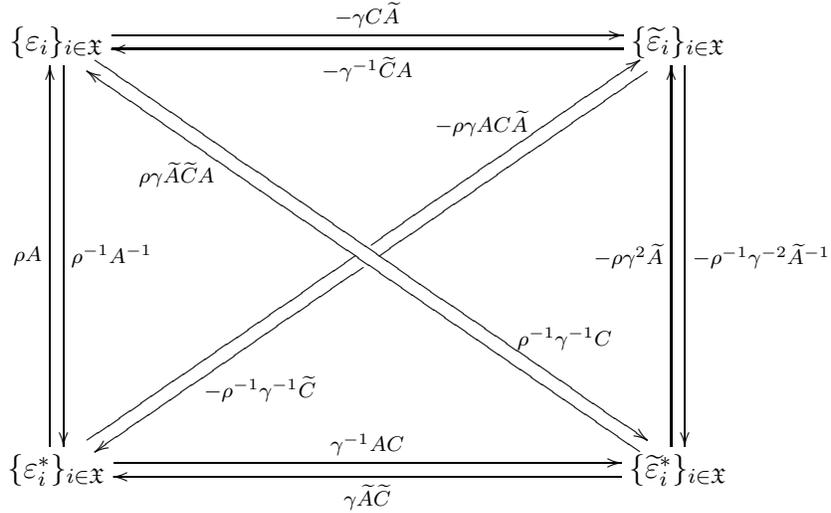
\begin{equation*}
\xymatrix@C=16pc@R=12pc{
\{\epsilon_i\}_{i \in \X}
        \ar@<.5ex>[r]^{-\gamma C\tilde{A}}
        \ar@<.5ex>[dr]^(.75){\rho^{-1}\gamma^{-1}C}
        \ar@<.5ex>[d]^{\rho^{-1}A^{-1}}
    & \{\tilde{\epsilon}_i\}_{i \in \X}
        \ar@<.5ex>[l]^{-\gamma^{-1}\tilde{C}A}
        \ar@<.5ex>[dl]^(.75){-\rho^{-1}\gamma^{-1}\tilde{C}}|\hole
        \ar@<.5ex>[d]^{-\rho^{-1}\gamma^{-2}\tilde{A}^{-1}}\\
    \{\epsilon_i^\ast\}_{i \in \X}
        \ar@<.5ex>[u]^{\rho A}
        \ar@<.5ex>[ur]^(.75){-\rho\gamma A C\tilde{A}}|\hole
        \ar@<.5ex>[r]^{\gamma^{-1}AC}
    & \{\tilde{\epsilon}_i^\ast\}_{i \in \X}
        \ar@<.5ex>[l]^{\gamma\tilde{A}\tilde{C}}
        \ar@<.5ex>[ul]^(.75){\rho\gamma\tilde{A}\tilde{C}A}
        \ar@<.5ex>[u]^{-\rho\gamma^{2}\tilde{A}}
}
\end{equation*}
where $\{u_i\}_{i \in \X} \displaystyle\mathrel{\mathop{\longrightarrow}^{\mathrm{B}}}\{v_i\}_{i \in \X}$ means $v_j = \sum_{i \in \X} B_{ij}u_i$ for all $j \in \X$.
\end{lemma}
\begin{proof}
Use Notation \ref{not:matrices} and verify using Theorem \ref{thm:transmat}, Lemma \ref{lem:dualtransmat}, Lemma \ref{lem:transmatbig}.
\end{proof}

\section{Cauchy matrices as transition matrices}\label{sec:mainresults}

In this section, we complete our description of Cauchy matrices associated to Cauchy pairs.

\begin{theorem}\label{thm:cmfromcp}
Let $(X, \tilde{X})$ be a Cauchy pair on $V$ and let $C \in {\rm Mat}_\X(\K)$ be a Cauchy matrix associated to $(X, \tilde{X})$.  Then there exists an $X$-eigenbasis $\{\epsilon_i\}_{i \in \X}$ for $V$ and an $\tilde{X}$-eigenbasis $\{\tilde{\epsilon}^\ast_i\}_{i \in \X}$ for $V$ such that $C$ is the transition matrix from $\{\epsilon_i\}_{i \in \X}$ to $\{\tilde{\epsilon}^\ast_i\}_{i \in \X}$.
\end{theorem}
\begin{proof}
By Definition \ref{def:assoc}, there exists eigenvalue data $(\{x_i\}_{i \in \X}, \{\tilde{x}_i\}_{i \in \X})$ for $(X, \tilde{X})$ such that $(\{x_i\}_{i \in \X}, \{\tilde{x}_i\}_{i \in \X})$ is data for $C$.  Let $\{\epsilon_i\}_{i \in \X}$ be an $X$-standard basis for $V$.  By Lemma \ref{lem:freeindex}, there exists an $\tilde{X}$-standard basis $\{\tilde{\epsilon}_i\}_{i \in \X}$ such that the index of $\{\epsilon_i\}_{i \in \X}, \{\tilde{\epsilon}_i\}_{i \in \X}$ is 1.  By Theorem \ref{thm:formexistance} and Lemma \ref{lem:norms}, there exists an $(X, \tilde{X})$-invariant form $\langle \ , \ \rangle$ on $V$ such that $||\epsilon_i||^2/\alpha_i = 1$ for all $i \in \X$.  Let $\{\tilde{\epsilon}^\ast_i\}_{i \in \X}$ denote the basis dual to $\{\tilde{\epsilon}_i\}_{i \in \X}$ with respect to $\langle \ , \ \rangle$.  Observe that $\{\tilde{\epsilon}^\ast_i\}_{i \in \X}$ is an $\tilde{X}$-eigenbasis for $V$ by Lemma \ref{lem:duals}. By Lemma \ref{lem:transmatbig}(i), the transition matrix from $\{\epsilon_i\}_{i \in \X}$ to $\{\tilde{\epsilon}^\ast_i\}_{i \in \X}$ has $(i,j)$-entry $(x_i - \tilde{x}_j)^{-1}$ for $i,j \in \X$.  By Definition \ref{def:data}, the transition matrix is equal to $C$.
\end{proof}

\begin{lemma}\label{lem:assoctoequiv}
Let $(X, \tilde{X})$ be a Cauchy pair on $V$ and let $C \in {\rm Mat}_\X(\K)$ be a Cauchy matrix associated to $(X, \tilde{X})$. Let $(Y, \tilde{Y})$ denote a Cauchy pair on $V$.  Suppose that $C$ is associated to $(Y, \tilde{Y})$.  Then $(X, \tilde{X})$ and $(Y, \tilde{Y})$ are equivalent.
\end{lemma}
\begin{proof}
By Theorem \ref{thm:cmfromcp}, there exists an $X$-eigenbasis $\{\epsilon_i\}_{i \in \X}$ for $V$ and an $\tilde{X}$-eigenbasis $\{\tilde{\epsilon}^\ast_i\}_{i \in \X}$ for $V$ such that $C$ is the transition matrix from $\{\epsilon_i\}_{i \in \X}$ to $\{\tilde{\epsilon}^\ast_i\}_{i \in \X}$.  Also by Theorem \ref{thm:cmfromcp}, there exists a $Y$-eigenbasis $\{\omega_i\}_{i \in \X}$ for $V$ and a $\tilde{Y}$-eigenbasis $\{\tilde{\omega}^\ast_i\}_{i \in \X}$ for $V$ such that $C$ is the transition matrix from $\{\omega_i\}_{i \in \X}$ to $\{\tilde{\omega}^\ast_i\}_{i \in \X}$.
Let $\phi \in \text{End}(V)$ be the vector space isomorphism that sends $\epsilon_i$ to $\omega_i$ for all $i \in \X$.  Since the transition matrix from $\{\epsilon_i\}_{i \in \X}$ to $\{\tilde{\epsilon}^\ast_i\}_{i \in \X}$ is equal to the transition matrix from $\{\omega_i\}_{i \in \X}$ to $\{\tilde{\omega}^\ast_i\}_{i \in \X}$, $\phi$ sends $\tilde{\epsilon}^\ast_i$ to $\tilde{\omega}^\ast_i$ for all $i \in \X$.

\medspace
By Definition \ref{def:assoc}, there exists eigenvalue data $(\{x_i\}_{i \in \X}, \{\tilde{x}_i\}_{i \in \X})$ for $(X, \tilde{X})$ and eigenvalue data $(\{y_i\}_{i \in \X}, \{\tilde{y}_i\}_{i \in \X})$ for $(Y, \tilde{Y})$ such that $(\{x_i\}_{i \in \X}, \{\tilde{x}_i\}_{i \in \X})$ and $(\{y_i\}_{i \in \X}, \{\tilde{y}_i\}_{i \in \X})$ are data for $C$.  By Lemma \ref{lem:equivdata}, there exists $\zeta \in \K$ such that $x_i = y_i + \zeta$ and $\tilde{x}_i = \tilde{y}_i + \zeta$ for all $i \in \X$.

\medspace
By Lemma \ref{lem:affinetrans}, $(Y + \zeta I, \tilde{Y} + \zeta I)$ is a Cauchy pair on $V$.  We show that $\phi$ is an isomorphism of Cauchy pairs from $(X, \tilde{X})$ to $(Y + \zeta I, \tilde{Y} + \zeta I)$. It is routinely checked that $\phi X$ and $(Y + \zeta I) \phi$ agree on $\epsilon_i$ for all $i \in \X$, so $\phi X = (Y + \zeta I) \phi$.  It is also routinely checked that $\phi \tilde{X}$ and $(\tilde{Y} + \zeta I)\phi$ agree on $\tilde{\epsilon}^\ast_i$ for all $i \in \X$, so $\phi \tilde{X} = (\tilde{Y} + \zeta I)\phi$. Thus by Definition \ref{def:iso}, $\phi$ is an isomorphism of Cauchy pairs from $(X, \tilde{X})$ to $(Y + \zeta I, \tilde{Y} + \zeta I)$.

\medspace
Therefore $(X, \tilde{X})$ is equivalent to $(Y, \tilde{Y})$ by Definition \ref{def:equivalence}.
\end{proof}

\begin{theorem}
Consider the map that sends a Cauchy pair $(X, \tilde{X})$ on $V$ to the set of Cauchy matrices in ${\rm Mat}_\X(\K)$ associated to $(X, \tilde{X})$. This map induces a bijection from the set of equivalence classes of Cauchy pairs on $V$ to the set of permutation equivalence classes of Cauchy matrices in ${\rm Mat}_\X(\K)$.
\end{theorem}
\begin{proof} Let $\phi$ denote the map that sends a Cauchy pair $(X, \tilde{X})$ on $V$ to the set of Cauchy matrices in ${\rm Mat}_\X(\K)$ associated to $(X, \tilde{X})$.
By Lemma \ref{lem:permequiv}, $\phi$ is a map from the set of Cauchy pairs on $V$ to the set of permutation equivalence classes of Cauchy matrices in ${\rm Mat}_\X(\K)$. The induced map on the set of equivalence classes of Cauchy pairs on $V$ is well-defined by Lemma \ref{lem:equivtoassoc}, injective by Lemma \ref{lem:assoctoequiv}, and surjective by Theorem \ref{thm:cpfromcm}.
\end{proof}

\section{Acknowledgements}

This paper was written while the author was a graduate student at the University of Wisconsin-Madison. The author would like to thank her advisor, Paul Terwilliger, for offering many valuable ideas and suggestions.

\bibliography{Cauchy_Pairs_and_Cauchy_Matrices}{}
\bibliographystyle{plain}

Alison Gordon Lynch \hfil\break
Department of Mathematics \hfil\break
University of Wisconsin \hfil\break
480 Lincoln Drive \hfil\break
Madison, WI 53706-1325 USA \hfil\break
email: {\tt gordon@math.wisc.edu }\hfil\break
\end{document}